\documentclass[12pt, a4paper]{amsart}

\usepackage{framed}
\usepackage{amsmath}
\usepackage{amssymb}
\usepackage{amsfonts}
\usepackage{mathrsfs}
\usepackage{amsthm}
\usepackage[all]{xy}
\usepackage{color}

\usepackage[top=30truemm,bottom=30truemm,left=30truemm,right=30truemm]{geometry}

\newcommand{\ra}{\rightarrow}

\newcommand{\hra}{\hookrightarrow}
\newcommand{\thra}{\twoheadrightarrow}
\newcommand{\lra}{\longrightarrow}

\newcommand{\id}{\mathop{\mathrm{id}}\nolimits}

\providecommand{\GL}[2]{{\rm GL}_{#1}(#2)}
\providecommand{\rep}[2]{\rho^{ }_{{#1},{#2}}}
\providecommand{\mrep}[2]{\bar{\rho}^{ }_{{#1},{#2}}}

\providecommand{\End}[2]{{\rm End^{ }}_{#1}(#2)}
\providecommand{\Gal}[2]{{\rm Gal}(#1/#2)}
\providecommand{\Ker}[1]{{\rm Ker}(#1)}

\providecommand{\Frob}[2]{ {\rm Frob}_{#1}^{#2 }}

\providecommand{\Fr}[2]{  \mathrm{Fr}_{ {#1} | {#2}  }  }
\providecommand{\FC}[3]{   \left[    \frac{{#1}/{#2}}{#3}\right] }

\makeatletter 
\@tempcnta\z@
\loop\ifnum\@tempcnta<26
\advance\@tempcnta\@ne
\expandafter\edef\csname rm\@Alph\@tempcnta\endcsname{\noexpand\mathrm{\@Alph\@tempcnta}}
\expandafter\edef\csname s\@Alph\@tempcnta\endcsname{\noexpand\mathscr{\@Alph\@tempcnta}}
\expandafter\edef\csname b\@Alph\@tempcnta\endcsname{\noexpand\mathbb{\@Alph\@tempcnta}}
\expandafter\edef\csname c\@Alph\@tempcnta\endcsname{\noexpand\mathcal{\@Alph\@tempcnta}}
\expandafter\edef\csname rm\@alph\@tempcnta\endcsname{\noexpand\mathrm{\@alph\@tempcnta}}
\repeat

\theoremstyle{plain}	
\newtheorem{thm}{Theorem}[section]
\newtheorem{prop}[thm]{Proposition}
\newtheorem{lem}[thm]{Lemma}
\newtheorem{conj}[thm]{Conjecture}

\newtheorem{prob}[thm]{Question}

\theoremstyle{definition}
\newtheorem{definition}[thm]{Definition}
\newtheorem{rem}[thm]{Remark}
\newtheorem{ex}[thm]{Example}

 \makeatletter
    \@addtoreset{equation}{section}
  \makeatother

\title{A function field analogue of the Rasmussen-Tamagawa conjecture: The Drinfeld module case}
\date{}
\author{Yoshiaki Okumura}

\begin{document}
\maketitle
\thispagestyle{empty}

\begin{abstract}
In the arithmetic of  function fields, Drinfeld modules play the role that elliptic curves play in the arithmetic of number fields.
The aim of this paper is to study a non-existence problem of Drinfeld modules with constrains on  torsion points at places with large degree.
This  is motivated by a conjecture 
of Christopher  Rasmussen and Akio  Tamagawa on the non-existence of  abelian varieties over number fields with some arithmetic constraints.
We prove the  non-existence of  Drinfeld modules satisfying Rasmussen-Tamagawa type conditions  in the case where the  inseparable degree of base fields is not   divisible by the rank of Drinfeld modules.
Conversely if the rank divides the inseparable degree, then  we give an example of Drinfeld modules satisfying Rasmussen-Tamagawa-type conditions.
\end{abstract}

\pagestyle{myheadings}
\markboth{Y.\ Okumura}{Non-existence of certain Drinfeld modules}

\renewcommand{\thefootnote}{\fnsymbol{footnote}}
\renewcommand{\thefootnote}{\arabic{footnote}}
\renewcommand{\thefootnote}{\fnsymbol{footnote}}
\renewcommand{\thefootnote}{\arabic{footnote}}


\section{Introduction}

Let $p$ be a  prime number  and  fix  some $p$-power $q=p^m$.
Write   $A:=\bF_q[t]$  for the polynomial ring in one variable $t$ over $\bF_q$ and set $F:=\bF_q(t)$.
Let $K$ be a finite extension of $F$.
In this article, we identify every monic irreducible element $\pi$ of $A$ with  the corresponding  finite place of $F$.
Write $\bF_\pi=A/\pi A$ for the residue field at $\pi$ and set $q_\pi:=\#\bF_\pi=q^{\deg (\pi)}$.

Let $r$ be a positive integer and $\pi$  a monic irreducible element of  $A$.
Define $\sD(K,r,\pi)$ to be the set of isomorphism classes $[\phi]$ of rank-$r$ Drinfeld modules
over  $K$ which satisfy the following two conditions:
\begin{itemize}
\setlength{\leftskip}{0.5cm}
\item[(D1)] $\phi$ has good reduction at any finite places of $K$ not lying above $\pi$,
\item[(D2)] the mod $\pi$ representation $\mrep{\phi}{\pi}:G_K \ra \GL{r}{\bF_\pi}$ attached to $\phi$
is of the form
\[
\mrep{\phi}{\pi} \simeq
\begin{pmatrix}
\chi_\pi^{i_1} & * & \cdots & *\\
 & \chi_\pi^{i_2}  & \ddots  & \vdots \\
 &  & \ddots  & * \\
 & &  &\chi_\pi^{i_{r}}  
\end{pmatrix},
\]
where $\chi_\pi$ is the mod $\pi$ Carlitz character (see Example \ref{Carlitz})  and $0 \leq i_1, \ldots,i_r \leq  q_\pi-1$ are integers.
\end{itemize} 
Consider the following:

\begin{prob}\label{conj}
Does there exist a positive constant $C>0$ depending only on $K$ and $r$ which satisfies the following:
if  $\deg(\pi)>C$, then  the set $\sD(K,r,\pi)$ is empty ?
\end{prob}

The motivation of this question is a non-existence conjecture on abelian varieties stated by Rasmussen and Tamagawa \cite{RT1}.
Let $k$ be a finite extension of $\bQ$ and $g$ a positive integer. 
For a prime number $\ell$, denote by $\tilde{k}_\ell$  the maximal pro-$\ell$ extension of $k(\mu_\ell)$ which is unramified outside $\ell$, where $\mu_\ell=\mu_\ell(\bar{k})$ is the set of $\ell$-th roots of unity.
For an abelian variety $X$ over $k$, write $k(X[\ell^\infty]):=k(\bigcup_{n \geq 1} X[\ell^n])$ for the field generated by all $\ell$-power torsion points of $X$.
Define  $\sA(k,g,\ell)$ to be  the set of isomorphism classes $[X]$ of $g$-dimensional abelian varieties over $k$ which satisfy the following equivalent conditions:
\begin{itemize}
\setlength{\leftskip}{0.5cm}
\item[(RT-1)]  $k(X[\ell^\infty]) \subseteq \tilde{k}_\ell$,
\item[(RT-2)] $X$ has good reduction at any finite place of $k$ not lying above $\ell$ and $k(X[\ell])/k(\mu_\ell)$ is an $\ell$-extension,
\item[(RT-3)]  $X$ has good reduction at any finite place of $k$ not lying above $\ell$ and 
the mod $\ell$ representation $\bar{\rho}_{X,\ell}:G_k\ra\GL{\bF_\ell}{X[\ell]}\simeq \GL{2g}{\bF_\ell}$ is of the form
\[
\bar{\rho}_{X,\ell} \simeq
\begin{pmatrix}
\chi^{i_1}_\ell & * & \cdots & *\\
 & \chi^{i_2}_\ell  & \ddots  & \vdots \\
 &  & \ddots  & * \\
 & &  &\chi^{i_{2g}}_\ell  
\end{pmatrix} ,
\]
where $\chi_\ell$ is the mod $\ell$ cyclotomic character.
\end{itemize}
These   conditions  come from the study of a question of Ihara \cite{Ihara} related with the kernel of the canonical outer Galois representation of the pro-$\ell$ fundamental group of 
$\bP^1 {\setminus} \{ 0,1,\infty\} $; see \cite{RT1}.
Rasmussen and Tamagawa conjectured the following:

\begin{conj}[Rasmussen and Tamagawa {\cite[Conjecture 1]{RT1}}] \label{RTconj}
The set $\sA(k,g,\ell)$ is empty  for any $\ell$ large enough.
\end{conj}

\noindent
Since the set $\sA(k,g,\ell)$ is always finite (see Section 5, or \cite{RT1}), the conjecture is equivalent to saying that the union $\bigcup_\ell \sA(k,g,\ell)$ is also finite.
For example, the following cases are known: 
\begin{itemize}
\item $k=\bQ$ and $g=1$ \cite[Theorem 2]{RT1},
\item $k=\bQ$ and $g=2,3$ \cite[Theorem 7.1 and 7.2]{RT2},
\item for abelian varieties with everywhere semistable reduction \cite[Corollary 4.5]{Ozeki} and \cite[Theorem 3.6]{RT2},
\item for abelian varieties with abelian Galois representations  \cite[Corollary 1.3]{OzekiCM},
\item for QM abelian surfaces over certain imaginary quadratic fields \cite[Theorem 9.3]{Arai}. 
\end{itemize}
We notice that,
under the assumption of the Generalized Riemann Hypothesis (GRH) for  Dedekind zeta functions of number fields, the conjecture is true in general \cite[Theorem 5.1]{RT2}.
The key tool of this proof is  the effective version of Chebotarev density theorem for number fields, which holds under GRH.
Rasmussen and Tamagawa  also state the ``uniform version'' of the conjecture  \cite[Conjecture 2]{RT2}, which says that one can take a lower bound of $\ell$ satisfying $\sA(k,g,\ell)=\varnothing$ depending only on the degree $[k:\bQ]$ and $g$.
For instance, the uniform version of the conjecture for CM abelian varieties is proved by Bourdon \cite[Corollary 1]{Bou} and Lombardo \cite[Theorem 1.3]{L}.
 Under GRH, the uniform version conjecture is true if $[k:\bQ]$ is odd \cite[Theorem 5.2]{RT2}.

The arithmetic properties of Drinfeld modules are similar to those of  elliptic curves over number fields.
Under this analogy, the  condition (D1)$+$(D2)  can be regarded as a natural translation of the condition (RT-3).
In fact, we  have also Drinfeld module versions of (RT-1) and (RT-2); 
we will show that the set  $\sD(K,r,\pi)$ is characterized  by  three equivalent conditions as in the abelian variety case (Proposition \ref{DRTcondi}).
The main purpose of this paper is to give a complete answer to Question \ref{conj}:

\begin{thm}[{Theorem \ref{main1} and Theorem \ref{main2}}]\label{MM}
If  $r$ does not divide the inseparable degree  $[K:F]_{\rm i}$ of $K/F$, then the set $\sD(K,r,\pi)$ is empty for any $\pi$ whose degree is large enough.
\end{thm}

\begin{thm}[{Theorem \ref{nonempty}}]\label{NE}
If $r$ divides $[K:F]_{\rm i}$, then the set $\sD(K,r,\pi)$ is never empty for any $\pi$. 
\end{thm}

The proof of Theorem \ref{MM} consists of  the two cases: (i) $r=p^\nu$, and (ii) $r=r_0\cdot p^{\nu}$ for some $r_0>1$ which is prime to $p$.
In the case (ii), 
we use the effective version of Chebotarev density theorem for function fields proved by Kumer and Scherk \cite{KS}, which is a modification of the strategy in \cite{RT2}.
In this case, the  uniform version, which is an analogue of \cite[Theorem 5.2]{RT2}, is also shown (Theorem \ref{uniform}).
However, the same argument  dose not work well in the case (i).
The proof in the case (i) is provided by observations about  the tame inertia weights of $\mrep{\phi}{\pi}$ 
for any $[\phi] \in \sD(K,r,\pi)$.
This technique is used in \cite{Ozeki} and \cite{RT2}.

There are difference between the number field setting and the function field setting.
Indeed if 
 $r$ divides $[K:F]_{\rm i}$ (for which there is no number field setting), then we construct a rank-$r$ Drinfeld module $\Phi$ over $K$ satisfying (D1) and (D2) for any $\pi$ and prove Theorem \ref{NE}.
This means that in this case a Drinfeld module analogue of the Rasmussen-Tamagawa conjecture is not true 
although the original  conjecture is generally true under GRH.

The organization of the paper is as follows.
In Section 2, after reviewing several basic facts on Drinfeld modules, we study the ramification of Galois representations coming from Drinfeld modules, whose consequences are needed  in the next section.
In Section 3, for any $[\phi] \in \sD(K,r,\pi)$, an important integer $e_\phi$   is introduced and 
we prove some non-trivial properties of it, which imply the result in the case (i).
The aim of Section 4 is to give the proof  in the case (ii).
For any  $[\phi] \in \sD(K,r,\pi)$, we introduce a character $\chi(m_\phi)$ and show the property that $\chi(m_\phi)$ never vanishes on the Frobenius elements of places with some conditions.
It contradicts  a consequence of the effective Chebotarev density theorem if $\deg(\pi)$ is sufficiently large.
Finally, in Section 5, we construct  a Drinfeld module satisfying both (D1) and (D2) for any $\pi$ 
in the case where $r$ divides $[K:F]_{\rm i}$.
We also show that the set $\sD(K,r,\pi)$ is infinite if $\pi=t$ and $r \geq 2$.

\subsection*{Notation}
Let $p,q,A,F$, and $K$ be as above. 
Set  $n_K^{}:=[K:F]$  and write $K_{\rm s}$ for the separable closure of $F$ in $K$.
For  a finite place $u$ of $K$ above $\pi$, let $K_u$ be the completion of $K$ at $u$, $\cO_{K_u}$ its valuation ring, and $\bF_u$ its residue field.
We use the same symbol $u$ for the normalized valuation of $K_u$. 
Set $q_u^{ }:= \# \bF_u$.
Identify $G_{K_u}$ with the decomposition group of $G_K$ at $u$ and regard it as a subgroup of $G_K$.
Denote by $I_{K_u}$  the inertia subgroup of $G_{K_u}$ at $u$ and choose a lift  $\Frob{u}{}  \in G_{K_u}$ of   the Frobenius element of  $G_{K_u}/I_{K_u}$.
Denote by $e_{u|\pi}$  the absolute ramification index of $u$ and  set $f_{u|\pi}:= [\bF_u:\bF_\pi]$.

Let $F_\infty:=\bF_q{(}{(}1/t))$ be the completion of $F$ at the infinite place $\infty$ of $F$ and $\bC_\infty$ the completion of a fixed algebraic closure of $F_\infty$.
Every algebraic extension of $F$ is always regarded as a subfield of $\bC_\infty$.
Let $|\cdot|$ be the absolute value of $F_\infty$ attached to the  normalized valuation of $F_\infty$.
We also denote by  $|\cdot|$ the unique extension of it  to $\bC_\infty$, which defines an absolute value of each algebraic extension of $F$ by restriction.
For any  non-zero $a \in A$, we see  that 
$|a|=\# (A{/}aA)=q^{\deg(a)} $.

For any field $L$, denote by $G_L:=\Gal{L^{\rm sep}}{L}$ the absolute Galois group of $L$. 
The notation $C=C(x,y,\ldots,z)$ indicates a constant $C$ depending only on
$x,y,\ldots,$ and $z$. 
We use the notation $\rho^{\rm ss}$ for the  semisimplification of a representation $\rho$.

\section*{Acknowledgments}
The author is grateful to his supervisor, Yuichiro Taguchi, for giving him useful advice about Drinfeld modules and for his guidance in preparing this paper.
The author is greatly indebted to Akio Tamagawa for pointing out a mistake in the preprint and for providing 
his idea to construct  examples in Propositions \ref{main3} and \ref{m}.
The author also would like to thank Yoshiyasu Ozeki for his helpful comments on Proposition \ref{DRTcondi}.  

\section{Drinfeld modules}
\subsection{Basic definitions}

Let  $L$ be a field  equipped with an $\bF_q$-algebra homomorphism  $\iota:A \ra L$.
Such a pair $(L,\iota)$ is called an {\it  $A$-field}.
Let  $\bG_{a,L}$ be  the additive group scheme defined over $L$.
Denote by $\End{\bF_q}{\bG_{a, L}}$ the ring of $\bF_q$-linear endomorphisms of $\bG_{a,L}$.
It is isomorphic to the non-commutative polynomial ring $L\{ \tau \}$ 
in one variable $\tau$ satisfying  $\tau c=c^{q}\tau$ for any $c \in L$, where $\tau$ is the $q$-power Frobenius map.
Let $r$ be a positive integer.

\begin{definition}\label{Dr}
A {\it Drinfeld module} $\phi$ of rank $r$ defined over the $A$-field  $L$ is an $\bF_q$-algebra homomorphism
\[
\phi:A \ra L \{ \tau \};a \mapsto \phi_a
\]
such that  $\phi_t=\iota(t) + c_1\tau + \cdots + c_r\tau^r$, where $c_1, \ldots ,c_r \in L$ and $c_r \neq 0$.
\end{definition}

Note that  $\phi$ is completely determined by the image $\phi_t$ of $t$.
For two Drinfeld modules $\phi$ and $\psi$  over $L$,
  a {\it morphism} from $\phi$ to $\psi$ is  an  element $f \in L\{ \tau \}$
such that  $f \phi_a = \psi_a  f$ for any $a \in A$.
We say that $f$ is an {\it isomorphism} if there exists a morphism $g$ from $\psi$ to $\phi$ such that $fg=gf=1$.
It is easy to check that $f$ is an isomorphism if and only if $f \in L^\times$.

For any $a \in A$, its image $\phi_a$ is an endomorphism of $\bG_{a,L}$, so that $\phi$ endows  the additive group $\bG_{\mathrm{a},L}(L^{\mathrm{sep}})\simeq L^{\mathrm{sep}}$ with a new $A$-module structure defined by $a\cdot \lambda:=\phi_a(\lambda)$.
Denote   this  $A$-module by ${}_{\phi}L^{\mathrm{sep}}$.
For any non-zero element $a \in A$, the set of {\it $a$-torsion points} 
\[
\phi[a]=\{ \lambda \in {}_{\phi}L^{\mathrm{sep}}; \ a\cdot \lambda = \phi_a(\lambda)=0 \}
\]
of $\phi$ is  an $A$-submodule of  ${}_{\phi}L^{\mathrm{sep}}$ on which  $G_L$ acts.
If $a$ is not contained in the kernel of $\iota$, then $\phi[a]$  is a free $A/aA$-module of rank $r$.

Let $K$ be a finite extension of $F$.
From now on, unless otherwise stated, we regard  $K$ as an  $A$-field via the inclusion $A \hra F \subset K$.
Let $\phi$  be a rank-$r$ Drinfeld module  over  $K$.
For any finite place  $v$ of $K$, 
 we can regard $\phi$ as a Drinfeld module over $K_v$ via  the canonical inclusion $K\{ \tau \} \hra K_v\{ \tau \}$.

\begin{definition}
(1) We say that $\phi$ has {\it stable reduction} at $v$  if there exists a Drinfeld module $\psi$ over $K_v$ such that $\psi$ is isomorphic to $\phi$ over $K_v$ and 
\[
\psi_t= t  + c'_1\tau+ \cdots + c'_r\tau^r
\]
such that  $c'_1 , \ldots , c'_r \in \cO_{K_v}$ and $ c'_{r'} \in \cO_{K_v}^\times$ for some $1 \leq r' \leq r$. 

(2) We say that $\phi$ has {\it good reduction} at $v$ if it has stable reduction  at $v$ and $c_r' \in \cO_{K_v}^\times$.
\end{definition}

\begin{prop}[Drinfeld {\cite[Proposition 7.1]{Dri}}]
Every Drinfeld module $\phi$ over $K$ has potentially stable reduction at any finite place $v$ of $K$.
\end{prop}
\begin{proof}
Write $\phi_t=t + c_1\tau + \cdots + c_r\tau^r$ and set $R:=\min_{1 \leq s \leq r} \{ \frac{v(c_s)}{q^s-1} \}$.
Let $K_v'$ be a finite extension of $K_v$.
If the ramification index $e(K_v'/K_v) $ satisfies  $e(K_v'/K_v) \cdot R \in \bZ$, then 
$\phi$ has stable reduction over $K_v'$.
\end{proof}

\begin{rem}\label{stablereduction}
In particular, we can take $K_v'/K_v$ as  a tamely ramified finite separable extension whose ramification index  $e(K_v'/K_v)$ divides $\prod_{s=1}^r(q^s-1)$.
Every rank-one Drinfeld module clearly has  potentially good reduction at any finite place.

\end{rem}

For any monic irreducible element $\pi \in A$, 
the set of $\pi$-torsion points $\phi[\pi]$ is a $G_K$-stable $r$-dimensional $\bF_\pi$-vector space, so that 
the {\it mod $\pi$ representation} 
\[
\mrep{\phi}{\pi}:G_K \ra \GL{\bF_\pi}{\phi[\pi]} \simeq \GL{r}{\bF_\pi}
\]
attached to $\phi$ can be defined.
Let   $A_\pi:=\varprojlim A/\pi^n A$ be the $\pi$-adic completion of $A$. 
Considering the maps $\phi[\pi^{n+1}] \ra \phi[\pi^n]$ defined by $x \mapsto \pi \cdot x$, one can define the {\it  $\pi$-adic Tate module} 
$T_\pi(\phi):= \varprojlim \phi[\pi^n]$, which is a free $A_\pi$-module of rank $r$ with continuous $G_K$-action.
Write $\rep{\phi}{\pi}: G_K \ra \GL{r}{A_\pi}$  for the representation attached to $T_\pi(\phi)$.

Let $\pi_0$  be a monic irreducible element of $A$ with $\pi_0 \neq \pi$ and let $v$ be a finite place of $K$ above $\pi_0$.
The next proposition is an analogue of the   N{\'e}ron-Ogg-Shafarevich criterion  for good reduction of  abelian varieties (cf.  \cite[Theorem 1]{ST}).

\begin{prop}[Takahashi {\cite[Theorem 1]{Tak}}]\label{NOS}
A Drinfeld module  $\phi$ over $K$ has good reduction at $v$ if and only if  $T_\pi(\phi)$ is unramified at $v$.
\end{prop}


Suppose that $\phi$ has good reduction at $v$. 
Then  $\rep{\phi}{\pi}$ is unramified at $v$, and so $\rep{\phi}{\pi}(\Frob{v}{}) \in \rep{\phi}{\pi}(G_K)$ is independent of the choice of $\Frob{v}{}$.
Denote by
\[
P_{v}(T):=\det(T-\rep{\phi}{\pi}(\Frob{v}{}) | T_\pi(\phi)) \in A_\pi[T]
\]
the characteristic polynomial of  $\Frob{v}{}$.
Then we have the following:

\begin{prop}[Takahashi {\cite[Proposition 3 (ii)]{Tak}}]\label{good}
The  polynomial $P_v(T)$ has coefficients in $A$ and is  independent  of $\pi$.
Any root  $\alpha$ of $P_v(T)$ satisfies  $|\alpha|=q_v^{1/r}$.
\end{prop}

The following example gives a function field analogue of cyclotomic extensions of number fields.

\begin{ex}[cf.\ {\cite[Chapter 12]{Ros}}]\label{Carlitz}
The rank-one Drinfeld module $\cC:A \ra F\{ \tau \}$ determined by $\cC_t= t + \tau$ is called the {\it Carlitz module}.
For any  monic irreducible element $\pi \in A$, define
\[
\chi_\pi:G_F \ra \GL{\bF_\pi}{\cC[\pi]} \simeq \bF_\pi^\times,
\]
which is called the {\it mod $\pi$ Carlitz character}.
Since $\cC$ has good reduction at any finite place $\pi_0$ of $F$, 
the  character $\chi_\pi$ is  unramified at  $\pi_0$ if $\pi_0 \neq \pi$.
For any finite place $v$ of $K$ above $\pi_0 \neq \pi$, it is known that  $\chi_\pi $ satisfies
\[
\chi_\pi(\Frob{v}{}) \equiv \pi_0^{f_{v|\pi_0}} \pmod \pi.
\]
The mod $\pi$ Carlitz character  induces an isomorphism $\Gal{F(\cC[\pi])}{F} \overset{\sim}{\ra} \bF_\pi^\times$
, so that $F(\cC[\pi])/F$ is a cyclic  extension which is unramified outside $\pi$ and $\infty$.
Moreover, it is known that $\pi$ is totally  ramified in $F(\cC[\pi])$
and the ramification of the infinite place $\infty$ is as follows: there exists a subfield $F(\cC[\pi])_+$ with degree $[F(\cC[\pi]):F(\cC[\pi])_+]=q-1$ such that $\infty$ is totally split in $F(\cC[\pi])_+$ and any place of $F(\cC[\pi])_+$ above $\infty$ is totally ramified in $F(\cC[\pi])$.

\end{ex}

\subsection{Tate uniformization}

Let $u$ be a finite place of $K$ above $\pi$ and $\phi$ a rank-$r$ Drinfeld module over $K$.
Suppose that $\phi$ has stable reduction at $u$.
Then Drinfeld's result on Tate uniformization gives an analytic description of  $\phi$ as a Drinfeld module over $K_u$.

\begin{prop}[Tate uniformization; Drinfeld {\cite[Section 7]{Dri}}]
There exist a unique  Drinfeld module  $\psi$ over $K_u$ with good reduction and a unique entire analytic surjective morphism $e:\psi \ra \phi$ defined over $K_u$ such that 
$e$ is the identity on $\mathrm{Lie}(\bG_{a,K_u})$.
\end{prop}

It is known that the rank $r'$ of $\psi$ satisfies $r' \leq r$ and the 
 kernel $H:=\Ker{e}(K_u^{\rm sep})$ is an $A$-lattice of rank $h:=r-r'$ in ${}_\psi K_u^{\rm sep}$, endowed with an action of a finite quotient of $G_{K_u}$.
For any monic irreducible element $\pi_0 \in A $, the analytic morphism $e$ induces the short exact sequence
\begin{align}\label{Tu}
0 \ra T_{\pi_0}(\psi) \ra T_{\pi_0}(\phi) \ra H\otimes_{A}A_{\pi_0} \ra 0
\end{align}
of $A_{\pi_0}[G_{K_u}]$-modules.
In the case where  $\pi_0 \neq \pi$,
 the $I_{K_u}$-action on $T_{\pi_0}(\phi)$ is potentially unipotent 
since both  $T_{\pi_0}(\psi)$ and $H\otimes_AA_{\pi_0}$ are potentially unramified at $u$.

\begin{rem}\label{tau}
By the theory on ``analytic $\tau$-sheaves''(see \cite{Gar}, \cite{Gar1} and  \cite{Gar2}), the sequence (\ref{Tu}) can be interpreted as follows.
For any Drinfeld module $\phi$ over $K_u$, one can construct the analytic $\tau$-sheaf 
$\tilde{M}(\phi)$ associated with $\phi$, which is a locally free $\sO_{\tilde{\bA}^1_{K_u}}$-module of finite rank  on $\tilde{\bA}^1_{K_u}$ with some additional structures, where $\tilde{\bA}^1_{K_u}$ is the affine line
$\bA_{K_u}^1=\mathrm{Spec}A \times_{\mathrm{Spec}\bF_q} \mathrm{Spec}K_u$, seen as a rigid analytic space. 
Then the $\pi_0$-adic Tate module $T_{\pi_0}(\tilde{M}(\phi))$ of $\tilde{M}(\phi)$ can be   defined and it is isomorphic to 
$T_{\pi_0}(\phi)$.
The Tate uniformization implies that there exist an analytic $\tau$-sheaf $\tilde{N}$ which is potentially trivial and the exact sequence
\[
0 \ra \tilde{N} \ra \tilde{M}(\phi) \ra \tilde{M}(\psi) \ra 0.
\]
Since $\tilde{M} \mapsto T_{\pi_0}(\tilde{M})$ is a contravariant exact functor, we obtain
\[
0 \ra T_{\pi_0}(\tilde{M}(\psi)) \ra T_{\pi_0}(\tilde{M}(\phi)) \ra T_{\pi_0}(\tilde{N}) \ra 0,
\]
which coincides with the sequence (\ref{Tu}) (for example, see \cite[Example 7.1]{Gar3}).
\end{rem}

We would like to estimate the tame ramification of the lattice $H$.
Suppose that $H$ is non-trivial, that is,  $h \neq 0$  and  consider the representation
\[
\rho : I_{K_u} \ra \GL{A}{H} \simeq \GL{h}{A}.
\]
Then we have the following:


\begin{prop}\label{stable}
There exists a finite separable extension $L/K_u$ such that
\begin{itemize}
\item the action of  $I_L$ on $H$ is trivial, 
\item the ramification index $e(L_0/K_u)$ divides $\prod_{s=1}^{r-1}(q^s-1)$, where $L_0$ is the maximal tamely ramified extension of $K_u$ in $L$.
\end{itemize}
\end{prop}

\begin{proof}
Let $E/K_u$ be a finite Galois extension such that the action of $G_{K_u}$ on $H$ factors through $\Gal{E}{K_u}$.
Now the image  $I$ of $I_{K_u}$ by the canonical projection map  $G_{K_u} \thra \Gal{E}{K_u}$ is the inertia subgroup of $\Gal{E}{K_u}$.
For the representation
\[
\rho:I \ra \GL{A}{H} \simeq \GL{h}{A},
\]
denote by $L$ the fixed subfield of $E$ by $J:=\Ker{\rho}$.
 By construction, the action of  inertia subgroup $I_L$ of $G_L$  on  $H$  is trivial. 
Now the image  $\mathrm{Im}(\rho)$ is a finite subgroup of $\GL{h}{A}$ of order $\#I/J=e(L/K_u)$ and $h \leq r-1$.
Applying Lemma \ref{gp} below to $\mathrm{Im}(\rho)$, we see that  $e(L_0/K_u)$ divides $\prod_{s=1}^{r-1}(q^s-1)$.
\end{proof}

\begin{rem}
 Proposition \ref{stable} means that the Drinfeld module $\phi$ is ``semistable'' over $L$ in the following sense. 
 By Remark \ref{tau}, the analytic $\tau$-sheaf $\tilde{M}(\phi)$  is  the extension 
  of  $\tilde{M}(\psi)$ by $\tilde{N}$   and both $\tilde{M}(\psi)$ and $\tilde{N}$  are ``good'' over $L$.
Hence the analytic $\tau$-sheaf $\tilde{M}(\phi)$ is  {\it strongly semistable} over $L$ in the sense of 
\cite[Definition 4.6]{Gar2}.
\end{rem}

\begin{lem}\label{gp}
For any positive integer $n$, 
let $G$ be a finite subgroup of $\mathrm{GL}_n(A)$.
Then the maximal prime-to-$p$ divisor of $\#G$ is a factor of  $\prod_{s=1}^{n}(q^s-1)$.
\end{lem}

\begin{proof}
Consider the $t$-adic completion $\bF_q[[t]]$ of $A$ and regard $G$ as a finite subgroup of $\GL{n}{\bF_q[[t]]}$.
Let $\Gamma_n$ be the kernel of the map $\GL{n}{\bF_q[[t]]} \thra \GL{n}{\bF_q}$ induced by the reduction map $\bF_q[[t]] \thra \bF_q$.
Since $\bF_q[[t]]$ is a complete noetherian local ring whose residue field is finite of characteristic $p$, $\Gamma_n$ is a pro-$p$ group.
Hence the short exact sequence $1 \ra \Gamma_n \ra \GL{n}{\bF_q[[t]]} \ra \GL{n}{\bF_q} \ra 1$ shows that the maximal prime-to-$p$ divisor of $\# G$ is a factor of $\# \GL{n}{\bF_q}=q^{n(n-1)/2}\prod_{s=1}^{n}(q^s-1)$.
\end{proof}

\section{Inertia action on torsion points}
Let $\pi$ be a monic irreducible element of $A$.
In this section, studying ramification of mod $\pi$ representations attached to  Drinfeld modules, we show the non-existence result (Theorem \ref{main1}) in the case where  $r$ is a $p$-power and  does not divide $[K:F]_{\rm i}$.

\subsection{Tame inertia weights}
Let $u$ be a finite place of $K$ above $\pi$.
For a fixed separable closure $K^{\rm sep}_u$ of $K_u$ with residue field $\bar{\bF}_u$, denote 
by $K_u^{\rm ur}$ (resp. $K_u^{\rm t}$)  the maximal unramified  (resp. maximal tamely ramified) extension of $K_u$ in  $K^{\rm sep}_u$, so that $I_{K_u}$ is isomorphic to $\Gal{K_u^{\rm sep}}{K_u^{\rm ur}}$. 
Denote by $I^{\rm t}:=\Gal{K_u^{\rm t}}{K_u^{\rm ur}}$  the tame inertia subgroup of $I_{K_u}$.
Let $d$ be a positive integer and $\bF$  the finite field   with $q_\pi^{d}$ elements in $\bar{\bF}_u$.
Then  $\bF$ is the finite extension of $\bF_\pi$ of degree $d$.
Write $\mu_{q_\pi^d-1}(K_u^{\rm sep})$ for the set of $(q_\pi^d-1)$-st roots of unity in $K^{\rm sep}_u$ and fix the  isomorphism $\mu_{q_\pi^d-1}(K_u^{\rm sep}) \overset{\sim}{\ra}  \bF^\times$ coming from the reduction map $\cO_{K_u^{\rm sep}} \thra \bar{\bF}_u$.
For a uniformizer $\varpi$ of $K_u$, choose a solution $\eta \in K_u^{\rm sep}$ to the equation 
$X^{q_\pi^d-1}- \varpi =0$ and define
\[
\omega_{d, K_u}:I_{K_u} \ra \mu_{q_\pi^d-1}(K_u^{\rm sep}) \overset{\sim}{\ra}  \bF^\times \ ; 
\sigma \mapsto \frac{\eta^\sigma}{\eta},
\]
which is independent of the choices of $\varpi$ and $\eta$.
The character $\omega_{d, K_u}$ factors through $I^{\rm t}$ (cf. \cite{Serre}).
We call the  $\Gal{\bF}{\bF_\pi}$-conjugates $(\omega_{d, K_u})^{q_\pi^i}$ for $0\leq i \leq d-1$ of $\omega_{d, K_u}$
  the {\it  fundamental characters} of level $d$.
It is easy to check that 
\[
(\omega_{d, K_u})^{1+q_\pi+\cdots + q_\pi^{d-1}} = \omega_{1, K_u}
\]
and $(\omega_{d, K_u})^{q_\pi^d-1}=1$.
For any  finite extension $L$ of $K_u$,
 we see that   $(\omega_{d, K_u})|_{I_L}=(\omega_{d,L})^{e(L/K_u)}$ by definition.

As an analogue of Serre's classical result on the mod $\ell$ cyclotomic character (\cite[Proposition 8]{Serre}), the following fact is known.

\begin{prop}[Kim {\cite[Proposition  9.4.3.\ (2)]{Kim}}]\label{fund}
The character $(\omega_{1, K_u})^{e_{u|\pi}}$ coincides with the mod $\pi$ Lubin-Tate character restricted  to $I_{K_u}$.
\end{prop}

\begin{rem}
The {\it  mod $\pi$ Lubin-Tate character} is the character describing  the $G_{K_u}$ action on the $\pi$-torsion points of   Lubin-Tate formal group over $\cO_{K_u}$ associated with  $\pi$. 
It coincides  with the mod $\pi$ Carlitz character $\chi_\pi$ restricted to $G_{K_u}$, so that $\chi_\pi=(\omega_{1, K_u})^{e_{u|\pi}}$ on $I_{K_u}$.
\end{rem}

Let $V$ be a $d$-dimensional irreducible $\bF_\pi$-representation of $I_{K_u}$.
Then the action of $I_{K_u}$ on $V$ factors through  $I^{\rm t}$,  so that  $V$ can be regarded as a representation of $I^{\rm t}$.
Using Schur's Lemma, we see that $\End{I^{\rm t}}{V}$ is a finite field of  order $q_\pi^d$.
Fix an isomorphism $f : \End{I^{\rm t}}{V} \overset{\sim}{\ra} \bF$ and regard $V$ as a one-dimensional $\bF$-representation
\[
\rho:I^{\rm t} \ra \End{I^{\rm t}}{V}^\times \overset{\sim}{\ra} \bF^\times
\]
of $I^{\rm t}$.
Since $I^{\rm t}$ is pro-cyclic and $\omega_{d, K_u}$ is surjective, there exists an integer 
$0 \leq j \leq q_\pi^{d}-2$ such that $\rho = (\omega_{d, K_u})^j$.
If we decompose $j=n_0 + n_1q_\pi + \cdots + n_{d-1}q_\pi^{d-1}$ with integers $0 \leq n_s \leq q_\pi-1$, then the set $\{ n_0, n_1,\ldots, n_{d-1} \}$ is independent of the choice of $f$.

\begin{definition}
These numbers $ n_0, n_1,\ldots,n_{d-1}$ are called {\it tame inertia weights} of $V$.
In general, for any $\bF_\pi$-representation $\rho: G_{K_u} \ra \GL{\bF_\pi}{V}$, the tame inertia weights of $\rho$ are the tame inertia weights of all the Jordan-H$\ddot{\rm o}$lder quotients of $V|_{I_{K_u}}$.

\end{definition}

Denote by $\mathrm{TI}_{K_u}(\rho)$ the set of tame inertia weights of $\rho:G_{K_u} \ra \GL{d}{\bF_\pi}$.

\subsection{Ramification of  constrained torsion points}
Let $\phi$ be a Drinfeld module over $K$ satisfying $[\phi] \in \sD(K,r,\pi)$
and $u$  a finite place of $K$ above $\pi$.
By Remark \ref{stablereduction}, we can take a finite separable extension $K_u'$ of $K_u$ such that $\phi$ has stable reduction and $e(K_u'/K_u)$ divides $\prod_{s=1}^r(q^s-1)$.
By  Tate uniformization, we obtain the exact sequence
\begin{align}\label{mrep}
0 \ra \psi[\pi] \ra \phi[\pi] \ra H{\otimes_A}\bF_\pi \ra 0
\end{align}
of $\bF_\pi[G_{K_u'}]$-modules.
We also take    a finite separable extension $L$ of $K_u'$ as in Proposition \ref{stable} and denote by $L_0$ the maximal tamely ramified extension of $K_u'$ in $L$.
Set 
\[
C_1=C_1(q,r):=(q^r-1)\prod_{s=1}^{r-1}(q^s-1)^2
\]
and $e_u:=e(L_0/K'_u) \cdot e(K_u'/F_\pi)$.
Then  $e_u$ divides $e_{u|\pi} C_1(q,r)$.

\begin{prop}\label{tame}
Every tame inertia weight of $\mrep{\phi}{\pi}|_{G_{L_0}}$ is between $0$ and $e_u$.
\end{prop}

\begin{proof}
By (D2), the restriction
$
\bar{\rho}_{\phi, \pi}^{\rm ss}|_{I_{L_0}}$ is isomorphic to $(\omega_{1, L_0})^{j_1} \oplus \cdots \oplus (\omega_{1, L_0})^{j_r}
$, where $\{ j_1, \ldots, j_r \}=\mathrm{TI}_{L_0}(\mrep{\phi}{\pi})$.
Write $\bar{\rho}:G_{L_0} \ra \GL{h}{\bF_\pi}$ for the representation arising from $H{\otimes}_A\bF_\pi$.
Then the sequence (\ref{mrep}) implies $\bar{\rho}_{\phi, \pi}^{\rm ss} = \bar{\rho}_{\psi, \pi}^{\rm ss} \oplus \bar{\rho}_{}^{\rm ss}$ on $G_{L_0}$ , so that 
$\mathrm{TI}_{L_0}(\mrep{\phi}{\pi}) = \mathrm{TI}_{L_0}(\bar{\rho}_{\psi,\pi}) \cup \mathrm{TI}_{L_0}(\bar{\rho})$.
Let $\tilde{M}(\psi)$ and $\tilde{N}$ be analytic $\tau$-sheaves on $\tilde{\bA}_{L_0}^1$ attached to $\psi$ and $H$, respectively.
Since  $\tilde{M}(\psi)$ is good over $L_0$,
 we see that $\mathrm{TI}_{L_0}(\bar{\rho}_{\psi,\pi}) \subset [0, e_u]$ by \cite[Theorem 2.14]{Gar}.
On the other hand,  the analytic $\tau$-sheaf $\tilde{N}$ is of dimension zero and  good over $L$, so that 
 every tame inertia weight of $\bar{\rho}|_{G_L}$ is zero by  \cite[Theorem 2.14]{Gar}, which means that $\bar{\rho}_{}^{\rm ss}|_{I_L}=1$.
Hence  we see that 
\[
(\omega_{1,L_0})^j|_{I_{L}}=(\omega_{1,L})^{e(L/L_0) \cdot j}=1
\]
for any $j \in \mathrm{TI}_{L_0}(\bar{\rho})$,
so that  $e(L/L_0) \cdot j \equiv 0 \pmod {q_\pi-1}$ holds.
Since $e(L/L_0)$ is a $p$-power and  $0 \leq j \leq q_\pi-2$, we see that   $j=0$.
\end{proof}

The condition (D2) means that 
$
\bar{\rho}_{\phi, \pi}^{\rm ss} $
is isomorphic to 
$\chi_\pi^{i_1} \oplus \cdots \oplus \chi_\pi^{i_r}$ for 
 $0 \leq i_1,\ldots, i_r \leq q_\pi-1$.
By renumbering  $\{ j_1, \ldots, j_r\}$ if necessary,  Propositions \ref{fund} and \ref{tame} mean   that $\chi_\pi^{i_s}|_{I_{L_0}}=(\omega_{1,L_0})^{ i_s\cdot e_u}=(\omega_{1,L_0})^{j_s}$ for any $1 \leq s \leq r$ .
Thus we obtain 
\begin{align}\label{ee}
i_s \cdot e_u \equiv j_s \pmod {q_\pi-1}
\end{align}
for any $1 \leq s \leq r.$

For any finite place $v$ of $K$ not lying above $\pi$ and any integer $m$, denote by
$
P_{v,m}(T)=\det (T-\rep{\phi}{\pi}(\Frob{v}{m}) | T_\pi(\phi)) \in A[T]
$
the characteristic polynomial of $\Frob{v}{m}$.
Set 
\[
C_2=C_2(n_K^{}, q,r):=r \cdot n_K^{2} \cdot C_1(q,r).
\]
Then we obtain the following important proposition.

\begin{prop}\label{cong}
If $\deg(\pi) > C_2$, then $r$ divides $e_u$ and the congruence
\[
 i_s \cdot e_u \equiv \frac{e_u}{r} \pmod {q_\pi-1}
\]
holds for  any  $1 \leq s \leq r.$
\end{prop}

\begin{proof}
Suppose  that $\deg(\pi)> C_2$.
Take a monic irreducible element $\pi_0 \in A $ with $\deg(\pi_0)=1$ and a finite place $v$ of $K$ above $\pi_0$.
Since   $\phi$ has good reduction at $v$ by (D1), the polynomial $P_{v,e_u}(T)$ is well-defined.
Now the  roots of $P_{v,e_u}(T)$ are given by $\{ \alpha_s^{e_u}\}_{s=1}^r$, where  $\{ \alpha_s \}_{s=1}^r$ are  the roots of $P_v(T)=P_{v,1}(T)$.
On the other hand, the condition (D2$)$ implies that the  roots of the polynomial
 $P_{v,e_u}(T) \pmod \pi$ in $\bF_\pi[T]$
are given by $\{ \chi_\pi(\Frob{v}{})^{i_s \cdot e_u} \}_{s=1}^r$.
Set $\pi_v:=\pi_0^{f_{v|\pi_0}}$.
By the above relation (\ref{ee}), we see that  $\chi_\pi(\Frob{v}{})^{i_s\cdot e_u}=\chi_\pi(\Frob{v}{})^{j_s}$
 for any $1 \leq s \leq r$.
Since $\chi_\pi(\Frob{v}{})^{j_s} \equiv \pi_v^{j_s} \pmod \pi$ holds by Example \ref{Carlitz}, we obtain
\begin{equation}\label{eigen}
P_{v,e_u}(T) \equiv \prod_{s=1}^r(T-\chi_\pi(\Frob{v}{})^{j_s}) \equiv \prod_{s=1}^r(T-\pi_v^{j_s}) \pmod \pi.
\end{equation}
Denote by  $S_k(x_1,\ldots,x_r)$  the fundamental symmetric polynomial of degree $k$ with
$r$ variables $x_1,\ldots,x_r$ for $0 \leq k \leq r$.
Then 
\[
\prod_{s=1}^r(T-x_s)=\sum_{k=0}^r(-1)^k S_k(x_1,\ldots,x_r)T^{r-k}.
\]
Now  $|\alpha_s^{e_u}| =q_v^{e_u/r}$ for any $1 \leq s \leq r$ by Proposition \ref{good}.
For any $0 \leq k \leq r$, we obtain

\begin{eqnarray}
\left|S_k(\alpha_1^{e_u},\ldots,\alpha_r^{e_u} )  - S_k(\pi_v^{j_1},\ldots,\pi_v^{j_r})\right| 
&\leq& \underset{1 {\leq} s_1 {<} {\cdots} {<} {s_k} {\leq} r}{\mathrm{max}} \left\{ q_v^\frac{k \cdot e_u}{r}, q_v^{j_{s_1}+\cdots+j_{s_k}} 
 \right\} \nonumber \\
&\leq&q_v^{k \cdot e_u} \nonumber \\ 
&\leq& q_v^{r \cdot e_u} =q^{r\cdot e_u \cdot f_{v|\pi_0}}_{}\nonumber
\end{eqnarray}
since $j_s \leq e_u$ for each $s$ by Proposition \ref{tame}. 
Since $e_u$ divides $e_{u|\pi}\cdot C_1(q,r)$ and both $e_{u|\pi}$ and $f_{v|\pi_0}$ are less than or equal to  $n_K=[K:F]$,
 we see that
\[
q^{r \cdot e_u \cdot f_{v|\pi_0}} \leq q^{C_2}<q^{\deg(\pi)}=|\pi|,
\]
which means that all absolute values of coefficients of  $P_{v,e_u}(T)-\prod_{s=1}^r(T-\pi_v^{j_s})$
are smaller than $|\pi|$.
Therefore the congruence  (\ref{eigen}) implies $P_{v,e_u}(T)=\prod_{s=1}^r(T-\pi_v^{j_s})$.
Comparing the absolute values of the roots of $P_{v,e_u}(T)$ and $\prod_{s=1}^r(T-\pi_v^{j_s})$,
 we see that
${e_u/r}=j_s$  for any $1 \leq s \leq r$, which implies the  conclusion.
\end{proof}

Set $e_{\phi}:={\gcd}\{ e_u ; \ u|\pi \}$ and $\bS_r:=\{ {\bf s}=(s_1, \ldots , s_r) \in \bZ^{r} ; 1 \leq s_k \leq r\}$.

\begin{lem}\label{e}
Suppose that $\deg(\pi)>C_2$.

$(1)$ $e_\phi| n_K^{} \cdot  C_1(q,r)$. If $\pi$ is unramified in $K_{\rm s}$, then $e_\phi | [K:F]_{\rm i} \cdot C_1(q,r)$.

$(2)$ For any 
$(s_1, \ldots , s_r) \in \bS_r$, the relation $e_\phi \cdot (i_{s_1} + \cdots + i_{s_r}-1) \equiv 0 \pmod {q_\pi-1}$ 
holds.
\end{lem}

\begin{proof}
Let $u$ be a finite place of $K$ above $\pi$ and $u_0$ the  place of $K_{\rm s}$ below $u$.
Then $e_{u|\pi}=e_{u_0|\pi}[K:F]_{\rm i}$ since $u_0$ is totally ramified in $K$ if $K \neq K_{\rm s}$.
Hence (1) follows from the relation $e_u|e_{u|\pi}C_1(q,r)$.
By Proposition \ref{cong}, we see that $i_s \cdot e_\phi \equiv \frac{e_\phi}{r} \pmod {q_\pi-1}$.
Adding this congruence for $s_1,\ldots,s_r$ gives
\[
e_\phi \cdot (i_{s_1}+ \cdots +i_{s_r}) \equiv e_\phi \pmod  {q_\pi-1},
\]
which proves (2).
\end{proof}

There exist only finitely many places of $F$ which are ramified in $K_{\rm s}$.
Define  $C_3=C_3(K_{\rm s})$ to be  the maximal degree of such places  and set
\[
C_4=C_4(n_K^{},q,r,K_{\rm s}) := \max \{C_2(n_K^{}, q,r),C_3(K_{\rm s})\}.
\] 

\begin{thm}\label{main1}
Suppose that  $r=p^\nu>1$ and $r$ does not divide $[K:F]_{\rm i}$.
If $\deg(\pi) > C_4$, then  the set $\sD(K,r,\pi)$ is empty.
\end{thm}

\begin{proof}
Assume that $[\phi] \in \sD(K,r,\pi)$ and $\deg(\pi) > C_4$.
Then $\pi$ is unramified in $K_{\rm s}$ by $\deg(\pi)>C_3$.
 Proposition \ref{cong} and Lemma \ref{e} imply that
$r $ divides $ [K:F]_{\rm i}\cdot C_1(q,r)$.
The integer $C_1(q,r)$ is prime to $p$ and so $r$ must divide $[K:F]_{\rm i}$, which contradicts  the assumption on $r$. 
\end{proof}

\begin{rem}
In Section 5, we see that if $r$ divides $[K:F]_{\rm i}$, then
$\sD(K,r,\pi)$ is not empty.
\end{rem}

\section{Observations at places with small degree}
In this section,
using Propositions \ref{mth} and \ref{mv}, we prove Theorem \ref{main2} on the emptiness  of $\sD(K,r,\pi)$ when $r$ is not a $p$-power.
In this case, we also prove its uniform version (Theorem \ref{uniform}).

\subsection{Effective Chebotarev density theorem} 
Recall some basic facts on function field arithmetic.
Let $L$ be an algebraic extension of $K$.
The {\it constant field}  $\bF_L$ of $L$ is the algebraic closure of $\bF_q$ in $L$.
If $L=\bF_LK$, then $L$ is called a {\it constant field extension} of $K$, which 
is unramified at any places (\cite[Proposition 8.5]{Ros}).
If $\bF_L=\bF_K$, then $L$ is called a {\it geometric extension} of $K$.
In general, the field $\bF_LK$ is the maximal constant extension of $K$ in $L$ and the extension $L/\bF_LK$  
 is geometric.
Set $[L:K]_{\rm g}:=[L:\bF_LK]$ if $L/K$ is finite, which is called the {\it geometric extension degree} of $L/K$.
For example, for any $a \in A$,  the field $F(\cC[a])$ arising from the Carlitz module is a geometric extension of $F$. 

Denote by $\mathrm{Div}(K)$ the divisor group of $K$, that is, the free abelian group generated by all places of $K$.
We write divisors additively, so that a typical divisor is of the form $D=\sum_vn_vv$.
The notation  $v \not\in D$ means that  $n_v=0$.
The {\it degree} of a place $v$ of $K$ is defined by  $\deg_Kv:=[\bF_v:\bF_K]$ and it is extended to any divisor $D=\sum_vn_vv$ by $\deg_KD=\sum_vn_v\deg_Kv$.
The degree $\deg_F\pi$
 of a finite place $\pi$ of $F$ is exactly the degree $\deg(\pi)$ as a polynomial. 

Suppose that $L$ is a finite separable extension of $K$.
Then the {\it conorm map} $i_{L/K}:\mathrm{Div}(K) \ra \mathrm{Div}(L)$ is defined to be the linear extension of 
\[
i_{L/K}v=\sum_{w|v}e_{w|v}w,
\]
 where     $v$ is a place of $K$.
The following is known (cf.\  \cite[Proposition 7.7]{Ros}).

\begin{lem}\label{deg}
Let $w$ be a place of $L$ above a place $v$ of $K$ and $D \in \mathrm{Div}(K)$.
Then
\[
\deg_Li_{L/K}D=[L:K]_{\rm g}\deg_KD \ \mbox{and}\ \deg_Lw=\frac{f_{w|v}}{[\bF_L:\bF_K]}\deg_Kv.
\]
\end{lem}

For any place $w$ of $L$ above a place $v$ of $K$, denote by $\mathfrak{p}_w$ the maximal ideal of $\cO_{L_w}$ and let $\delta_w$ 
the exact power
 of $\mathfrak{p}_w$ dividing the different of $\cO_{L_w}$ over $\cO_{K_v}$. 
Then it satisfies $\delta_w \geq e_{w|v}-1$ with equality holding if and only if $p$ 
 does not divide $e_{w|v}$.
Define  the {\it ramification divisor} of $L/K$ by
$
\cD_{L/K}=\sum_w\delta_ww.$
For any intermediate field $K'$ of $L/K$, we see that 
\[
\cD_{L/K}=\cD_{L/K'} + i_{L/K'}\cD_{K'/K}
\]
(for example, see \cite[Chapter III 4]{Serre2}).
Hence  $\cD_{L/K'} \leq \cD_{L/K}$ holds.
 In addition, the following holds (cf.\  {\cite[Lemma 2.6]{CL}}).

\begin{lem}\label{d}
Let $L/K$ and $L'/K$ be finite separable extensions.
Then
\[
\cD_{LL'/K} \leq i_{LL'/L}\cD_{L/K} + i_{LL'/L'}\cD_{L'/K}.
\]
\end{lem}

Now let $E$ be a finite Galois extension of $K$ and $v$ a place of $K$ unramified in $E$.
For any place $w$ of $E$ above $v$, denote by $\mathrm{Fr}_{w|v}$ the Frobenius element in  $ \Gal{E}{K}$.
These  elements consist a conjugacy class 
\[
\FC{E}{K}{v}:=\{ \Fr{w}{v} ;   w|v\ \}
\]
in $\Gal{E}{K}$.
Define $\Sigma_{E/K}$ to be the divisor of $K$ that is the sum of all ramified places of $K$ in $E$.
 As a consequence of  the effective version of Chebotarev density theorem \cite[Theorem 1]{KS}, the following holds.

\begin{prop}[Chen and Lee {\cite[Corollary 3.4]{CL}}]\label{CL}
Let $E/K$ be a finite Galois extension and $\Sigma$ a divisor of $K$ such that 
$\Sigma \geq \Sigma_{E/K}$.
Set $d_0:=[\bF_K:\bF_q]$ and $d:=[\bF_E:\bF_K]$.
Define the constant $B=B(E/K, \Sigma)$ by 
\[
B=\max\{\deg_K\Sigma, \deg_E\cD_{E/\bF_EK}, 2[E:\bF_EK]-2, 1 \}.
\]
Then for any nonempty conjugacy class $\sC$ in $\Gal{E}{K}$, there exists a place $v$ of $K$ 
with $v \notin \Sigma$ such that 
\begin{itemize}
\item $\sC=\FC{E}{K}{v}$, 
\item $\deg_Kv \leq \frac{4}{d_0} \log_q\frac{4}{3}(B+3g_K+3) + d$,
\end{itemize}
where $g_K$ is the genus of $K$.
\end{prop}


Let $\pi$ be a monic irreducible element of $A$ and $m \geq 1$ an integer
which divides $\#\bF_\pi^\times = q_\pi-1$.
A monic irreducible element $\pi_0$ distinct from  $\pi$ is called an {\it $m$-th power residue modulo $\pi$} if $(\pi_0 \ \mbox{mod}\ \pi) \in (\bF_\pi^\times)^m$.
As an application of Proposition \ref{CL}, we show that  one can find an $m$-th power residue modulo $\pi$ whose degree is  smaller than  $\deg(\pi)$ if $\deg(\pi)$ is sufficiently large.
Denote by $F_m$ the unique  subfield of $F(\cC[\pi])$ with $[F_m:F]=m$ and 
consider the character
$
\chi(m):G_F \overset{\chi_\pi}{\lra} \bF_\pi^\times \thra \bF_\pi^\times/(\bF_\pi^\times)^m.
$

\begin{lem}\label{est}
The following are equivalent.
\begin{itemize}
\item  $\pi_0$ is an $m$-th power residue modulo $\pi$.
\item  $\chi(m)(\Frob{\pi_0}{})=1$.
\item  $\Frob{\pi_0}{}{\mid}_{F_m}^{}=\id$.
\end{itemize}
\end{lem}

\begin{proof}
It is trivial when $m=1$.
If not, then this lemma immediately follows from  that $\chi_\pi(\Frob{\pi_0}{}) \equiv \pi_0\ (\mbox{mod}\ \pi)$ and $F_m$ is the fixed field of $F^{\rm sep}$ by $\Ker {\chi(m)}$.
\end{proof}

Denote by $\tilde{K}$ the Galois closure of $K_{\rm s}$ over $F$
and set $E:=\tilde{K}F_m$, which is also a Galois extension of $F$.
Consider the divisor  $\Sigma:=\Sigma_{E/F} +\pi + \infty \in \mathrm{Div}(F)$.
For the constant  $B=B(E/F, \Sigma)=\max\{\deg_F\Sigma, \deg_E\cD_{E/\bF_EF}, 2[E:\bF_EF]-2, 1 \}$, we obtain the following estimate.

\begin{lem}\label{B}
Let $n$ be a positive integer.
Then there exists a constant $C_5=C_5(K_{\rm s}, n_K^{}, q,m , n)>0$ such that 
for any $\pi$ satisfying $\deg(\pi) > C_5$, the inequality 
\[
4 \log_q\frac{4}{3}(B + 3) + [\bF_{\tilde{K}}: \bF_q] < \frac{1}{n} \deg(\pi)
\]
holds.
\end{lem} 
\begin{proof}
We first compute an upper  bound of $B$.
We may assume that $\pi$ is unramified in $K_{\rm s}$. 
Since the degree $[\tilde{K}:F]$ is less than or equal to $n_{K}^{}!$, we see that 
$[\bF_{\tilde{K}}:\bF_q] \leq n_K^{}!$ and $[E:\bF_EF] \leq m\cdot n_K^{}!$.
By Example \ref{Carlitz}, the infinite place $\infty$ of $F$ is split into at most $m$ places in $F_m$
whose ramification indices divide $q-1$ and $\pi$ is totally ramified or unramified  (if $m=1$) in $F_m$.
Thus we see that 
\[
\deg_F\Sigma \leq \deg_F(\Sigma_{F_m/F} + \Sigma_{\tilde{K}/F} + \pi + \infty)\leq 2\deg(\pi)+2+\deg_F\Sigma_{\tilde{K}/F}.
\]
Now $\cD_{E/\bF_EF} = \cD_{E/F}$ holds.
Lemmas \ref{deg} and \ref{d} imply
\begin{eqnarray}
\deg_E\cD_{E/F} &{\leq}& \deg_{E}i_{E/\tilde{K}}\cD_{\tilde{K}/F} + \deg_{E}i_{E/F_m}\cD_{F_m/F} \nonumber \\
&\leq& m\cdot \deg_{\tilde{K}}\cD_{\tilde{K}/F} + [E:F_m]_{\rm g}
\cdot \deg_{F_m}(\sum_{v |\infty}(q-2)v + m\pi) \nonumber \\
&\leq& m\cdot \deg_{\tilde{K}}\cD_{\tilde{K}/F} +  n_K! \cdot m(q-2 + \deg(\pi)). \nonumber
\end{eqnarray}
Hence there exist positive constants $B_1$ and $B_2$ depending only on $K_{\rm s}, n_K$, $q$ and $m$ such that 
$
B \leq B_1 \cdot \deg(\pi) +B_2
$ holds.
Therefore if $\deg(\pi)$ is sufficiently large, then $4 \log_q\frac{4}{3}(B + 3) + [\bF_{\tilde{K}}: \bF_q] < \frac{1}{n} \deg(\pi)$ holds.
\end{proof} 

 Proposition \ref{CL} and Lemma \ref{B} imply the following:

\begin{prop}\label{mth}
Let $n$ be  a positive integer. 
If $\deg(\pi) > C_5$, then there exist a monic irreducible element $\pi_0 \in A$
and a place $v$ of $K$ above $\pi_0$ such that 
\begin{itemize}
\item $\pi_0$ is an $m$-th power residue modulo $\pi$,
\item $\deg(\pi_0) < \frac{1}{n}\deg(\pi)$,  
\item $f_{v|\pi_0}=1$.
\end{itemize}
\end{prop}

\begin{proof}

We may assume that $K=K_{\rm s}$ since the extension $K/K_{\rm s}$ is totally ramified at any place  if $K \neq K_{\rm s}$.
Let $\tilde{K}$ and $E=\tilde{K}F_m$ be as above and fix an element $\sigma \in \Gal{E}{F}$ such that $\sigma|_{KF_m}=\id$.
For the conjugacy class $\sC$ of $\sigma$ in $\Gal{E}{F}$, by Proposition \ref{CL} and Lemma \ref{B},
there exists a  place $\pi_0$ of $F$ with $\pi_0 \notin \Sigma$ (hence it is a finite place) such that 
$\FC{E}{F}{\pi_0}=\sC$ and $\deg(\pi_0) < \frac{1}{n}\deg(\pi)$, so that $\sigma=\mathrm{Fr}_{w|\pi_0}$
for some place $w$ of $E$.
Then the decomposition group $Z_w$ of $w$ over $\pi_0$ is generated by $\sigma$ and  
it is a  subgroup of $\Gal{E}{KF_m}$.
Denote by $K'$ the fixed field of $E$ by $Z_w$.
Then the  place $v'$ of $K'$ below $w$ satisfies
$e_{v'|\pi_0}=1$ and $f_{v'|\pi_0}=1$.
Hence  $f_{v|\pi_0}=1$, where $v$ is the place  of $K$ below $v'$. 
By construction, we see that   $\Frob{v}{}|_{F_m}=\id$.
Lemma \ref{est} means that $\pi_0$ is an
$m$-th power residue modulo $\pi$.
\end{proof}

\subsection{Non-$p$-power rank case}
Let $\phi$ be a rank-$r$ Drinfeld module over $K$ satisfying $[\phi] \in \sD(K,r,\pi)$.
In this subsection, we always assume that $r=r_0 \cdot p^\nu$ for some  $r_0>1$ which   is prime to $p$. 
Now let $i_1, \ldots, i_r$ be the positive integers satisfying  $\bar{\rho}^{\rm ss}_{\phi,\pi} \simeq \chi_\pi^{i_1} \oplus \cdots \oplus \chi_\pi^{i_r}$ by (D2). 
For any  ${\bf s}=(s_1, \ldots, s_r) \in \bS_r$, set $\varepsilon_{\bf s}:=\chi_\pi{}^{i_{s_1}+\cdots +i_{s_r}-1}$ and define 
\[
\epsilon:=(\varepsilon_{\bf s})_{{\bf s} \in \bS_r} : G_F \ra (\bF_\pi^\times)^{\oplus r^r}.
\]
Set $m_\phi:=\#\epsilon(G_F)$, which is the least common multiple of the orders of  $\varepsilon_{\bf s}$.
Since $\epsilon$ factors through $\bF_\pi^\times$, the image $\epsilon(G_F)$ is cyclic and $m_\phi|(q_\pi-1)$.
Then we obtain the following commutative diagram

\begin{equation}
\begin{xy}
(30,20) *{\epsilon(G_F)}="G",
(0,0) *{G_F }="H", (2,-3)*{ }="S", (30,0) *{\bF_\pi^\times }="I", (60,0)*{ }="Q", (60,3)*{ }="P",(60,-3)*{ }="R",(70,0) *{\bF_\pi^\times /(\bF_\pi^\times)^{m_\phi}}="J".

\ar "H";"G"^-{\epsilon}
\ar@{>>} "I";"G"^-{}
\ar "H";"I"^-{\chi_\pi}
\ar@{>>}"I";"Q"^-{}
\ar "P";"G"_{\simeq}
\ar @/_2mm/@{.>}"S";"R"_{\chi \left(m_\phi \right)}
\end{xy} .\nonumber
\end{equation}

\noindent
Hence a monic irreducible element $\pi_0$ is an $m_\phi$-th power residue modulo $\pi$ if and only if $\varepsilon_{\bf s}(\Frob{\pi_0}{})=1$ for any ${\bf s} \in \bS_r$.

\begin{lem}
If $\deg(\pi)>C_2(n_K,q,r)$, then $m_\phi$ divides the greatest common divisor $(e_\phi, q_\pi-1)$.
In particular, it divides $n_K C_1(q,r)$.
\end{lem}

\begin{proof}
It follows from Lemma \ref{e} (2). 
\end{proof}

\begin{prop}\label{mv}
If there exist a monic irreducible element $\pi_0$ and a finite place $v$ of $K$ above $\pi_0$ such that $\deg(\pi) > f_{v|\pi_0} \deg(\pi_0)$ and $r_0$ does not divide $f_{v|\pi_0}$,
then $m_{\phi}>1$ and $\chi(m_\phi)(\Frob{v}{}) \neq 1$.
\end{prop}

\begin{proof}
Assume that either $m_\phi=1$ or $\chi(m_\phi)(\Frob{v}{})=1$ holds.
Then  $\varepsilon_{\bf s}(\Frob{v}{})=1$ for any ${\bf s} \in \bS_r$.
Denote by $a_{v,p^\nu} \in A$ the coefficient of $T^{r-p^\nu}$ in the characteristic polynomial $P_v(T)$ of $\Frob{v}{}$ on $T_\pi(\phi)$.
It is given by $a_{v,p^\nu}=(-1)^{p^\nu}S_{p^\nu}(\alpha_1, \ldots, \alpha_r)$, where $\alpha_1, \ldots, \alpha_r$ are the roots of $P_v(T)$ and $S_{p^\nu}(x_1, \ldots,x_r)$ is the fundamental symmetric polynomial of degree $p^\nu$ with $r$ variables.
Consider the subset  $\bS_{r,p^\nu}:=\{ (s_1, \ldots, s_{p^\nu}) ; 1 {\leq} s_1 {<}\cdots {<}s_{p^\nu} {\leq} r\}$ of $\bZ^{p^\nu}_{}$. 
Then the product $\bS_{p^\nu}^{r_0}$ can be  regarded as a subset of $\bS_r$.
Since $S_{p^\nu}(x_1, \ldots,x_r)$ is the sum of $\binom{r}{p^\nu}$ monomials of degree $p^\nu$, we obtain that 
\begin{eqnarray}
(a_{v,p^\nu})^{r_0} {=} (-1)^r S_{p^\nu}(\alpha_1, \ldots, \alpha_r)^{r_0} &\equiv&
 (-1)^r \left(\sum_{(s_1, \ldots, s_{p^\nu}) \in \bS_{r,p^\nu}} \chi_\pi^{i_{s_1}+\cdots + i_{s_{p^\nu} }}(\Frob{v}{}) \right)^{r_0}   \nonumber \\
&\equiv&  (-1)^r\sum_{{\bf s} \in \bS_{p^\nu}^{r_0}} \varepsilon_{\bf s}(\Frob{v}{})\chi_{\pi}(\Frob{v}{}) \nonumber \\
&\equiv&  (-1)^r\sum_{{\bf s} \in \bS_{p^\nu}^{r_0}} \chi_\pi(\Frob{v}{}) \nonumber \\
&\equiv&  (-1)^r\binom{r}{p^\nu}^{r_0} \pi_0^{f_{v|\pi_0}}  \pmod \pi \nonumber 
\end{eqnarray}
and $ (-1)^r\binom{r}{p^\nu}^{r_0} \pi_0^{f_{v|\pi_0}} \not\equiv 0 \pmod \pi$ since $\binom{r}{p^\nu}$ is not divisible by  $p$.
Now we see that 
\[
|(a_{v,p^\nu})^{r_0}| \leq  q_v =q^{f_{v|\pi_0}\deg(\pi_0)} < |\pi|\ 
\mbox{and}\ 
\left| (-1)^r\binom{r}{p^\nu}^{r_0} \pi_0^{f_{v|\pi_0}} \right|=|\pi_0^{f_{v|\pi_0}}| =q_v < |\pi|.
\]
Hence the above congruence implies $(a_{v,p^\nu})^{r_0}=(-1)^r\binom{r}{p^\nu}^{r_0} \pi_0^{f_{v|\pi_0}}$.
Comparing  the $\pi_0$-adic valuations of both sides, we obtain $r_0|f_{v|\pi_0}$, which is a contradiction. 
\end{proof}
Set 
\begin{eqnarray}
C_6 &=& C_6(K_{\rm s}, n_K, q,r):=\max\{ C_5(K_{\rm s}, n_K, q,m,1) ;\  m|n_KC_1(q,r)\} \nonumber \\
C_7 &=& C_7(K_{\rm s}, n_K,q, r):=\max\{ C_2(n_K, q,r), C_6(K_{\rm s}, n_K,q,r)\}. \nonumber
\end{eqnarray}
Then we have the following theorem.

\begin{thm}\label{main2}
Suppose that $r=r_0p^\nu$  as above and $\deg(\pi)>C_7$.
 Then the set  $\sD(K,r,\pi)$ is empty.  
\end{thm}

\begin{proof}
Assume that $\sD(K,r,\pi)$ is not empty and $[\phi] \in \sD(K,r,\pi)$.
By Proposition \ref{mth}, there exist a monic irreducible element $\pi_0$ and a place $v$ of $K$ above $\pi_0$
such that $f_{v|\pi_0}=1$, $\deg(\pi_0) < \deg(\pi)$ and $\chi(m_\phi)(\Frob{\pi_0}{})=1$.
However, since  $\pi_0$ and $v$ satisfy the assumption of Proposition \ref{mv}, we see that 
$\chi(m_\phi)(\Frob{v}{})=\chi(m_\phi)(\Frob{\pi_0}{}) \neq 1$.
\end{proof}

By the same argument, we can also prove a uniform version.
For a fixed finite separable extension $K_0$ of $F$ with degree $n_0:=[K_0:F]$ and  a positive integer $n$, set
\[
C_8=C_8(K_0,q, r, n):=\max \left\{ C_2(n n_0, q,r), \max\{C_5(K_0,n_0,q, m, n);\ m | n_0C_1(q,r) \} \right\}.
\]

\begin{thm}\label{uniform}
Let $r=r_0p^\nu$, $K_0$, and $n$ be as above.
Suppose that $r_0$ does not divide $n$.
If $\deg(\pi)>C_8$, 
then  for any finite extension $K$ of $K_0$ satisfying  $[K:K_0]=n$, the set 
$
\sD(K,r,\pi)
$ is empty, namely the union
\[
\underset{[K:K_0]=n}{\bigcup}\sD(K,r,\pi)
\]
is empty.
\end{thm}

\begin{proof}
Let $K$ be  a finite extension of $K_0$ with $[K:K_0]=n$ and assume that  $[\phi] \in \sD(K,r,\pi)$.
Applying Proposition \ref{mth} to $K_0$, we can find a monic irreducible element $\pi_0$ and a finite place $v_0$ of $K_0$ above $\pi_0$ such that  $f_{v_0|\pi_0}=1$, $n\deg(\pi_0) < \deg(\pi)$ and $\chi(m_\phi)(\Frob{\pi_0}{})=1$.
Now we can take  a place $v$ of $K$ above $v_0$ such that $f_{v|v_0} (=f_{v|\pi_0})$ is not divisible by $r_0$.
Indeed, if not, then $r_0$ must divide $n=\sum_{v|v_0}e_{v|v_0}f_{v|v_0}$. 
Since $f_{v|\pi_0}\deg(\pi_0) < n\deg(\pi_0)< \deg(\pi)$, by Proposition \ref{mv}, we see that
 $\chi(m_\phi)(\Frob{v}{})=\chi(m_\phi)(\Frob{\pi_0}{})^{f_{v|\pi_0}} \neq 1$.
It is a  contradiction.
\end{proof}

\section{Comparison with number field case}
In this final section, 
we compare the Rasmussen-Tamagawa conjecture and its Drinfeld module analogue.
After studying the similarly of them, we construct an example of a  Drinfeld module satisfying Rasmussen-Tamagawa type conditions for any $\pi$ and prove Theorem \ref{nonempty}.
We also prove  the infiniteness of $\sD(K,r,\pi)$ for $r\geq 2$ and $\pi=t$ in Proposition \ref{infinite}.  

\subsection{Defining conditions of $\sD(K,r,\pi)$}
In number field case,  $\sA(k,g,\ell)$ is  defined by the  equivalent conditions (RT-1), (RT-2), and (RT-3) in Section 1.
The equivalence of them follows  from the criterion of N{\'e}ron-Ogg-Shafarevich  and the next group theoretic lemma:
\begin{lem}[Rasmussen and Tamagawa {\cite[Lemma 3.4]{RT2}}]\label{group}
Let $\bF$ be a finite field of characteristic $\ell$.
Suppose $G$ is a profinite group, $N \subset G$ is a pro-$\ell$ open normal subgroup, and 
$C=G{/}N$ is a finite cyclic subgroup with $\#C | \#\bF^\times$.
Let $V$ be an $\bF$-vector space of dimension $r$ on which $G$ acts continuously.
Fix a group homomorphism $\chi_0:G \ra \bF^\times$ with $\mathrm{Ker}(\chi_0)=N$.
Then there exists a filtration
\[
0=V_0 \subset V_1 \subset \cdots \subset V_r=V
\]
such that each $V_s$ is $G$-stable and $\dim_\bF V_s=s$ for  any $0 \leq s \leq r$.
Moreover, for each $1 \leq s \leq r $, the $G$-action on each quotient $V_s{/}V_{s-1}$ is given by
$\chi_0^{i_s}$ for some integer $i_s$ satisfying $0 \leq i_s < \#C$.
\end{lem}
\begin{rem}
In \cite{RT2}, this lemma is proved when $\bF=\bF_\ell$.
The general case can be proved in the same way.
\end{rem}

As an analogue, we give two conditions which are equivalent to (D1)$+$(D2).
Let $\phi$ be a rank-$r$ Drinfeld module over $K$ and let $\pi \in A$ be a monic irreducible element.
Consider the field  $K(\phi[\pi^\infty]):=K({\bigcup}_{n \geq 1}\phi[\pi^n])$  generated by all $\pi$-power torsion points of $\phi$, so that it coincides with the fixed subfield 
of $K^{\rm sep}$ by  the kernel of $\rep{\phi}{\pi}:G_K \ra {\rm GL}_r(A_\pi)$.
Recall that the mod $\pi$ Carlitz character $\chi_\pi:G_K \ra {\rm GL}_{\bF_\pi}(\cC[\pi]) \simeq \bF_\pi^\times$ is an analogue of the mod $\ell$ cyclotomic character.
For the field $L:=K(\phi[\pi]) \cap K(\cC[\pi])$, we can prove the next proposition in the same way  as the abelian variety case.

\begin{prop}\label{DRTcondi}
The following conditions are equivalent.
\begin{itemize}
\setlength{\leftskip}{0.5cm}
\item[(DR-1)]  $K(\phi[\pi^\infty])/L$ is a pro-$p$ extension which is unramified at any finite place of $L$ not lying above $\pi$,
\item[(DR-2)] $\phi$ has good reduction at any finite place of $K$ not lying above $\pi$ and $K(\phi[\pi])/L$ is a $p$-extension,
\item[(DR-3)] $\phi$ satisfies {\rm (D1)} and {\rm (D2)}.
\end{itemize}
\end{prop} 

\begin{rem}
Unlike the abelian variety case, the field $K(\phi[\pi])$ may not contain $K(\cC[\pi])$.
For example, for $x \in \bF_q^\times {\setminus} \{  1 \}$, consider the rank-one Drinfeld module 
$\phi$ over $F$ determined by $\phi_t=t + x\tau$ and suppose $q \neq 2$.
Then the fields $F(\phi[t])$ and $F(\cC[t])$ are generated by the roots of $t + xT^{q-1}$ and $t + T^{q-1}$, respectively.
By Kummer theory, we see that $F(\phi[t]) \neq F(\cC[t])$,  so that  $F(\phi[t]) \not\supset F(\cC[t])$.
\end{rem}

\begin{proof}[Proof of Proposition \ref{DRTcondi}]
Since the kernel of ${\rm GL}_r(A_\pi) {\thra} {\rm GL}_r(\bF_\pi)$ is a pro-$p$ group, the extension $K(\phi[\pi^\infty])/K(\phi[\pi])$ is always pro-$p$.
The extension  $K(\cC[\pi])/K$ is unramified at any finite place of $K$ not lying above $\pi$ 
(Example \ref{Carlitz}).
Hence the conditions (DR-1) and (DR-2) are equivalent by Proposition \ref{NOS}.
Suppose that (DR-2) holds.
Then the condition  (DR-3) immediately follows from  Lemma \ref{group} for $G={\rm Gal}(K(\phi[\pi])/K)$, $\chi_0=\chi_\pi|_G$, 
$N=\mathrm{Ker}(\chi_\pi|_G)={\rm Gal}(K(\phi[\pi])/L)$, and $V=\phi[\pi]$.
Conversely, if (DR-3) holds, then the image $\mrep{\phi}{\pi}(N)$ of $N={\rm Gal}(K(\phi[\pi])/L)$
is contained in 
\[
\left\{
\begin{pmatrix}
1 &  & {\large*}\\
 & \ddots & \\
 & & 1
\end{pmatrix}
\in {\rm GL}_r(\bF_\pi) \right\},
\]
which is a Sylow $p$-subgroup of ${\rm GL}_r(\bF_\pi)$.
Since $\mrep{\phi}{\pi}|_G$ is injective, we see that  $K(\phi[\pi])/L$ is a $p$-extension.
\end{proof}

\begin{rem}
The original conjecture of Rasmussen and Tamagawa is formulated for abelian varieties of arbitrary dimension, and so we would like to formulate its function field analogue for  some higher dimensional objects (recall that Drinfeld modules are analogues of elliptic curves). 
In \cite{Anderson}, Anderson introduced objects called {\it $t$-motives} as analogues of abelian varieties of higher dimensions, which are also generalizations of Drinfeld modules.
In fact the category of Drinfeld modules is anti-equivalent to that of $t$-motives of dimension one.
It is known that $t$-motives have the notions of good reduction and Galois representation attached to their $\pi$-torsion points (see, for example \cite{Gar}), so that we can consider the conditions (D1) and (D2) for $t$-motives.
Moreover, Proposition \ref{DRTcondi} is also generalized to $t$-motives since Galois criterion of good reduction for $t$-motives holds. 
Therefore  the set $\sM(K,d,r,\pi)$ of isomorphism classes of $d$-dimensional $t$-motives over $K$ of rank $r$ satisfying the Rasmussen-Tamagawa type conditions can be defined and the following question makes sense: is the set $\sM(K,d,r,\pi)$ empty for any $\pi$ with sufficiently large degree?
\end{rem}

\subsection{Non-emptiness of $\sD(K,r,\pi)$}

In this subsection, giving a concrete example, we prove the following theorem:
\begin{thm}\label{nonempty}
If $r$ divides  $[K:F]_{\rm i}$, then the set $\sD(K,r,\pi)$ is never empty for any $\pi$.
\end{thm}

If $r=1$, then Theorem \ref{nonempty} is trivial since the Carlitz module $\cC$ satisfies both (D1) and (D2).
Assume that $r \geq 2$ and $[K:F]_{\rm i}$ is divisible by $r$, so that $r$ is a $p$-power. 
Now
the $r$-power map $A \ra A;a \mapsto a^r$ is an injective ring homomorphism.

For any $a =\sum x_nt^n \in A$ with $x_n \in \bF_q$, set $\hat{a}:=\sum x_n^{1/r}t^n$.
Then we see that $a \mapsto \hat{a}$ is a ring automorphism of $A$ and that the  map
$A \ra A; a \mapsto \hat{a}^r$ is an injective $\bF_q$-algebra homomorphism.

\begin{lem}\label{ps}
Set $[K:F]_{\rm i}=p^\nu$. Then $K_{\rm s}=K_{}^{p^{\nu}}$.
\end{lem} 

\begin{proof}

Since $K$ is a purely inseparable extension of $K_{\rm s}$ of degree $p^{\nu}$, the field 
$K_{}^{p^{\nu}}$ is contained in  $K_{\rm s}$.
Consider the sequence of fields $K \supset K^p \supset \cdots \supset K^{p^{\nu}} $.
Proposition 7.4 of \cite{Ros} implies that   each extension $K^{p^n}/K^{p^{n+1}}$
 is of degree $p$.
Hence $[K:K_{}^{p^{\nu}}]=p^{\nu}=[K:K_{\rm s}]$, which means that $K_{\rm s}=K_{}^{p^{\nu}}$.
\end{proof}

Since $r$ divides $[K:F]_{\rm i}$, Lemma \ref{ps} implies that $K$ contains the field 
$F^{{1/r}}$.
In particular the $r$-th root $t^{1/r}$ of $t$ is contained in $K$, so that 
we have a new injective $A$-field structure $\iota:A \ra K$ defined by $\iota(t)=t^{1/r}$.
Define the rank-one Drinfeld module 
\[
\cC':A \ra K\{ \tau \}
\]
 over the   $A$-field $(K,\iota)$ by 
$\cC'_t=t^{1/r}+\tau$.

Set ${}^{(r)}\mu:=\sum c_i^r \tau^i$ for any  $\mu=\sum c_i\tau^i \in K\{ \tau \}$, which defines a ring homomorphism $K\{ \tau \} \ra K\{ \tau \}$.
Then we can relate $\cC'$ with the Carlitz module $\cC$ as follows:

\begin{lem}\label{relation}
{\rm (1)} For any $a \in A$, ${}^{(r)}\cC'_{\hat{a}}=\cC_a$. 

{\rm (2)} For any  element $\lambda \in \cC'[\hat{a}]$, there exists a unique   $\delta \in \cC[a]$ such that $\lambda=\delta^{1/r}$.
\end{lem} 

\begin{proof}
Clearly ${}^{(r)}\cC'_{\hat{x}}=x=\cC_x$ for any $x \in \bF_q$ and 
${}^{(r)}\cC'_{\hat{t}}={}^{(r)}\cC'_{t}=\cC_t$.
Hence for any $a=\sum x_nt^n$, 
\[
{}^{(r)}\cC'_{\hat{a}}={}^{(r)}\left( \sum x_n^{1/r} (\cC'_t)^n \right)=\sum x_n (\cC_t)^n=\cC_a.
\]
For any $\lambda \in \cC'[\hat{a}]$, we see that 
\[
0=\left(\cC'_{\hat{a}}(\lambda) \right)^r={}^{(r)}\cC'_{\hat{a}}(\lambda^r)=\cC_a(\lambda^r),
\]
so that $\lambda^r \in \cC[a]$ and  we have the injective homomorphism $\cC'[{\hat{a}}] \ra \cC[a];\lambda \mapsto \lambda^r$ of finite groups.
Since $\#\cC'[\hat{a}] $ is equal to $\#\cC[a]$  by $\deg(\hat{a})=\deg(a)$, it is a bijection.
\end{proof}

Now define $\Phi_a:=\cC'_{\hat{a}^r}=(\cC'_{\hat{a}})^r \in K\{ \tau  \}$ for any $a \in A$.
Then
 by construction it gives  an $\bF_q$-algebra homomorphism
\[
\Phi: A \ra K\{ \tau \}, 
\]
which is determined by $\Phi_t=(t^{1/r} + \tau)^r$.
Since $\iota(\hat{a}^r)=a$ holds, $\Phi$ is  a rank-$r$ Drinfeld module over $K$.
Moreover it has good reduction at every finite place $v$ of $K$ since $v(t^{1/r}) \geq 0$.

By the following proposition, we see that $[\Phi] \in \sD(K,r,\pi)$, which implies Theorem \ref{nonempty}.

\begin{prop}\label{main3}
Let $i$ be the positive integer satisfying $ir \equiv 1 \pmod {q_\pi-1}$ and $i < q_\pi-1$.
Then the mod $\pi$ representation attached to $\Phi$ is of the form
\[
\mrep{\Phi}{\pi}\simeq 
\begin{pmatrix}
\chi_\pi^i &  & {\large*}\\
 & \ddots & \\
 & & \chi_\pi^i
\end{pmatrix}.
\]
\end{prop}

\begin{proof}
It suffices to prove that   $\bar{\rho}_{\Phi,\pi}^{\rm ss} = (\chi_\pi^i)^{\oplus r}$.
For each  $1 \leq s \leq r$,  set
\[
V_s:=\{ \lambda \in {}_\Phi K^{\rm sep} ; \cC'_{\hat{\pi}^{s}}(\lambda)=0\}.
\]
For any $a \in A$ and $\lambda \in V_s$, we see that  $\Phi_a(\lambda) \in V_s$ since 
$
\cC_{\hat{\pi}^s}'(\Phi_a(\lambda))= \cC_{\hat{\pi}^s}'(\cC_{\hat{a}^r}'(\lambda))=\cC_{\hat{a}^r}'(\cC_{\hat{\pi}^s}'(\lambda))=0.
$
Hence
$V_s$ is an $A$-submodule of ${}_\Phi K^{\rm sep}$ with the natural $G_K$-action.
Moreover  $\Phi_\pi(\lambda)=0$ for any $\lambda \in V_s$, so that $V_s$ is an $\bF_\pi$($=A/\pi A$)-vector space.
Here  $\Phi[\pi]=V_r$ by the definition of $\Phi$.
Then we obtain  the filtration 
\[
0 =V_0 \subset V_1 \subset V_2 \subset \cdots \subset V_r=\Phi[\pi]
\]
of  $G_K$-stable $\bF_\pi$-subspaces of $\Phi[\pi]$.
Now  the map  $V_s\ra V_1;\lambda \mapsto \cC'_{\hat{\pi}^{s-1}}(\lambda)$ induces a $G_K$-equivariant isomorphism $V_s{/}V_{s-1} \cong V_1$.
Since  $V_1=\cC'[\hat{\pi}]$ (as a set) and  $\deg(\pi)=\deg(\hat{\pi})$, we have 
$\#V_1=q_{\hat{\pi}}=q_\pi (=\# \bF_\pi)$.
Hence $\dim_{\bF_\pi}V_1=1$ and 
 the semisimplification of $\Phi[\pi]$ (as an $\bF_\pi[G_K]$-module) is 
$\Phi[\pi]^{\rm ss}  =\oplus_{s=1}^rV_s{/}V_{s-1} \cong V_1^{\oplus r}$.
For any  $\sigma \in G_K$ and $\lambda \in V_1$, we prove $\sigma(\lambda)=\chi_\pi(\sigma)^i \cdot \lambda$ as follows.
Take an element $a_\sigma \in A$ satisfying $a_\sigma \equiv \chi_\pi(\sigma) \pmod \pi$.
By Lemma \ref{relation} (2),  $\lambda=\delta^{1/r}$ for some $\delta \in \cC[\pi]$. 
Then 
\[
\sigma(\lambda)^r=\sigma(\delta)=\chi_{\pi}(\sigma)\cdot\delta=\cC_{a_\sigma}(\delta).
\]
Now the $\bF_\pi$-vector space structure of $V_1$ is determined by  $\Phi$ and so 
$\chi_\pi(\sigma)^i \cdot \lambda = \Phi_{a_\sigma^i}(\lambda) = \cC'_{\hat{a}_\sigma^{ir}}(\lambda)$.
Since 
$ir \equiv 1 \pmod {q_{\hat{\pi}}-1}$ holds, 
 we have
 $\hat{a}_{\sigma}^{ir}  \equiv \hat{a}_{\sigma}   \pmod {\hat{\pi}}  $.
This implies
 $\cC'_{\hat{a}_\sigma^{ir}}(\lambda) =\cC'_{\hat{a}_\sigma}(\lambda)$.
By Lemma \ref{relation} (1), we obtain
\[
\left( \chi_\pi(\sigma)^i\cdot\lambda \right)^r
=\left( \cC'_{\hat{a}_\sigma}(\lambda) \right)^r
={}^{(r)}\cC'_{\hat{a}_\sigma}(\lambda^r)
=\cC_{a_\sigma}(\delta)=\sigma(\lambda)^r.
\]
Since the $r$-power map is injective,  we have
$\sigma(\lambda)=\chi_\pi(\sigma)^i \cdot \lambda$.
Hence 
 the $G_K$-action on $V_1$ is given by $\chi_\pi^i$.
\end{proof}


\begin{rem}
Let  $u$ be a finite place of  $K$ above $\pi$. Now $r$ divides $e_{u|\pi}$ by assumption and set $e=e_{u|\pi}/r$.
Since $ir \equiv 1 \pmod {q_\pi-1}$, we see that 
\[
\chi_\pi^i|_{I_{K_u}}=(\omega_{1,K_u})^{i \cdot e_{u|\pi}} = (\omega_{1,K_u})^{i \cdot r \cdot e}=(\omega_{1,K_u})^e.
\]
Hence the set of tame inertia weights of $\mrep{\Phi}{\pi}|_{I_{K_u}}$ is $\mathrm{TI}_{{K_u}}(\mrep{\Phi}{\pi})=\{e\}$.
\end{rem} 

\subsection{Infiniteness of $\sD(K,r,t)$}
Finally, for $\pi=t$, we construct an infinite subset of $\sD(K,r,t)$.
In number field case, the set $\sA(k,g,\ell)$ is always finite because of the Shafarevich conjecture proved by Faltings \cite{Fal},  which states that there exist only finitely many  isomorphism classes of abelian varieties over fixed $k$ with fixed dimension $g$ which have good reduction outside a fixed finite set of finite places of $k$.
However, the Drinfeld module analogue of it  does not hold:

\begin{ex}\label{a}
For any $a\in A$, consider  the  rank-2 Drinfeld module $\phi^{(a)}:A \ra F\{ \tau \}$  given by 
$
\phi^{(a)}_t=t + a\tau + \tau^2.
$
It is  easily seen that   $\phi^{(a)}$ has good reduction at any  finite place of $F$.
If $\phi^{(a)}$ is isomorphic to  $\phi^{(a')}$ for some $a' \in A$ over $F$,
then there exists an element $c \in F$ such that $c\phi_t^{(a')} = \phi_t^{(a)}c$, so that 
\[
\phi_t^{(a')}=t + a'\tau + \tau^2=t + c^{q-1}a\tau + c^{q^2-1}\tau^2.
\]
This means that $c \in \bF_q^\times$ and hence $a'=c^{q-1}a=a$.
Therefore the set of isomorphism classes 
$
\{ [\phi^{(a)}] ; a \in A \}
$ is infinite.
\end{ex}

Let $W$ be a $G_K$-stable one-dimensional $\bF_q$-vector space contained  in  $K^{\rm sep}$ and write 
$\kappa_W: G_K \ra \bF_q^\times$ for the character attached to $W$.
Set
 $P_W(T):=\prod_{\lambda \in W}(T-\lambda)$, which is an $\bF_q$-linear  polynomial of the form
\[
P_W(T)=T^q+c_WT,  \ \ \ c_W:= \Bigl({\prod}_ {\lambda \in W\setminus \{0 \}}-\lambda \Bigr) \in K^\times
\]
by \cite[Corollary 1.2.2]{Goss}.
For any $c \in K^\times$, denote by $\bar{c} \in K^\times{/}(K^\times)^{q-1}$ the class of $c$ and by 
$\kappa_{(c)}:G_K \ra \bF_q^\times$ the character corresponding to  $\bar{c}$ by the map 
$K^\times{/}(K^\times)^{q-1} \overset{\sim}{\ra} \mathrm{Hom}(G_K,\bF_q^\times)$ of Kummer theory.

\begin{lem}\label{k}
For the above element  $c_W \in K^\times$, the character $\kappa_{(-c_W)}$ coincides with $\kappa_W$.
\end{lem}

\begin{proof}
Since $\lambda^{q-1}=-c_W$ for any $\lambda \in W{\setminus}\{0\}$, the character $\kappa_{(-c_W)}$ is given by
$
\kappa_{(-c_W)}(\sigma)={\sigma(\lambda)/\lambda}=\kappa_W(\sigma)
$
for any $\sigma \in G_K$.
\end{proof} 

Identify $\bF_t=A/tA=\bF_q$.
Then $\cC[t]$ is a one-dimensional $\bF_q$-subspace of   $K^{\rm sep}$ and   $P_{\cC[t]}(T)=T^q + t T$ by the definition of $\cC$.
By  Lemma \ref{k}, we see that $\chi_t=\kappa_{(-t)}$.
Note that $\chi_t^i=\kappa_{((-t)^i)}$ for any integer $i$.

Take $r$ elements $c_1, \ldots, c_r \in K^\times$.
For any  $1 \leq s \leq r$, define 
$f_s(\tau):=(\tau + c_s)(\tau +c_{s-1}) \cdots (\tau+c_1) \in K\{ \tau \}$
and set $W_s:=\Ker {f_s:K^{\rm sep} \ra K^{\rm sep} }$, which is a $G_K$-stable $s$-dimensional $\bF_q$-subspace of $K^{\rm sep}$.
Thus we obtain the filtration
\[
0=W_0 \subset W_1 \subset \cdots \subset W_r
\]
of $\bF_q[G_K]$-modules. 

\begin{lem}\label{c}
The $\bF_q$-linear representation $\bar{\rho}:G_K \ra \GL{\bF_q}{W_r} \simeq \GL{r}{\bF_q}$ is of the form
\[
\bar{\rho} \simeq
\begin{pmatrix}
\kappa_{(-c_1)} & * & \cdots & *\\
 & \kappa_{(-c_2)}  & \ddots  & \vdots \\
 &  & \ddots  & * \\
 & &  &\kappa_{(-c_{r})}  
\end{pmatrix}.
\]
\end{lem}

\begin{proof}
For any $1 \leq s \leq r$, the quotient $W_s{/}W_{s-1}$ is isomorphic to $\Ker {\tau + c_s:K^{\rm sep} \ra K^{\rm sep}}$ as an $\bF_q[G_K]$-module.
Hence  each $W_s{/}W_{s-1}$ is embedded into $K^{\rm sep}$.
By Lemma \ref{k}, the action of $G_K$ on $W_s{/}W_{s-1}$ is given by $\kappa_{(-c_s)}$.
\end{proof}

Fix $r$ integers $i_1, \ldots , i_r$ satisfying   $\sum_{s=1}^r i_s=1$.
For any ${\bf m}=(m_1, \ldots , m_r) \in \bZ^{r}$ satisfying   $\sum_{s=1}^rm_s=0$, 
consider the $\bF_q$-algebra homomorphism $\phi^{\bf m}:A \ra K\{ \tau \}$ given by
\[
\phi^{\bf m}_t=(-1)^{r-1}\prod_{s=1}^r(\tau - (-t)^{k_s}),
\]
where $k_s=i_s + m_s(q-1)$ for any $1 \leq s \leq r$.
Now  $\sum_{s=1}^rk_s=1$, so that the constant term of $\phi_t^{\bf m}$ is  $(-1)^{r-1}\prod_{s=1}^r( -(-t)^{k_s})=(-1)^{2r}t=t$ and hence 
$\phi^{\bf m}$ is 
a rank-$r$ Drinfeld module  over $K$.

\begin{prop}\label{m}
The isomorphism class $[\phi^{\bf m}]$ is contained in $\sD(K,r,t)$.
Moreover,  
the mod $t$ representation attached to $\phi^{\bf m}$ is of the form
\[
\mrep{\phi^{\bf m}}{t} \simeq 
\begin{pmatrix}
\chi_t^{i_1} & * & \cdots & *\\
 & \chi_t^{i_2}  & \ddots  & \vdots \\
 &  & \ddots  & * \\
 & &  &\chi_t^{i_r}  
\end{pmatrix},
\]
where $i_1, \ldots, i_r$ are the integers fixed as above.
\end{prop}

\begin{proof}
For any   finite place  $v$ of $K$ not lying above $t$,
since  $-t \in \cO_{K_v}$ and the leading coefficient of $\phi_t^{\bf m}$ is $(-1)^{r-1}$, we see that $\phi^{\bf m}$ has good reduction at $v$.
Now $\phi^{\bf m}[t]$ coincides with the kernel of
$\prod_{s=1}^r(\tau-(-t)^{k_s})$.
Applying Lemma \ref{c} to $f_s=(\tau-(-t)^{k_s})\cdots(\tau-(-t)^{k_1})$, we see that $\mrep{\phi^{\bf m}}{t}$ is given as above since $\kappa_{((-t)^{k_s})}=\chi_t^{k_s}=\chi_t^{i_s}$ for any $1 \leq s \leq r$.
\end{proof}

\begin{prop}\label{infinite}
If $r \geq 2$, then the set $\sD(K,r,t)$ is infinite.
\end{prop}
\begin{proof}
We construct an infinite subset of $\sD(K,r,t)$ as follows.
Fix $r$ integers  $i_1, \ldots, i_r$  satisfying $\sum_{s=1}^ri_s=1$.
For any positive integer $m$, consider  $(-m,0,\ldots,0,m) \in \bZ^r$ and  define $\phi^m:=\phi^{(-m,0,\ldots,0,m)}$, which is a Drinfeld module satisfying $[\phi^m] \in \sD(K,r,t)$ by Proposition \ref{m}.
Write  $\phi^m_t=t + c_1\tau + \cdots + c_{r-1}\tau^{r-1} + (-1)^{r-1}\tau^r \in K\{ \tau \}$. Then by construction the coefficient $c_{r-1}$ is given by 
\[
c_{r-1}=(-t)^{i_1-m(q-1)} +  (-t)^{i_r +m(q-1)} + \sum_{s=2}^{r-1}(-t)^{i_s}.
\]
For any finite place $u$ of $K$ above $t$, if $m$ is sufficiently large, then 
\[
u(c_{r-1})=(i_1-m(q-1)) u(-t)<0,
\]
hence we see that $u(c_{r-1}) \ra -\infty$ as $m \ra \infty$.
On the other hand, for two positive integers $m$ and $m'$,  if $\phi^{m'}$ is isomorphic to $\phi^m$, then $\phi^{m'}_t=x^{-1}\phi^m_tx$ for some $x \in \bF_K^\times$ by the same argument of Example \ref{a}.
These facts imply that  if $m'$ is sufficiently large, then $\phi^{m}$ and $\phi^{m'}$ are not isomorphic.
Therefore the subset $\{[\phi^m]; m \in \bZ_{>0} \}$ of $\sD(K,r,t)$ is infinite.
\end{proof}

\ \\

Department of Mathematics, Tokyo Institute of Technology

2-12-1 Oh-okayama, Meguro-ku, Tokyo 152-8551, Japan

{\it  E-mail address} : \email{\bf okumura.y.ab@m.titech.ac.jp}

\begin{thebibliography}{9999}

 
\bibitem[An]{Anderson} G.\ W.\ Anderson,
{\it $t$-motives}, 
Duke  Math.\ J.\ {\bf 53}, 457--502 (1986)


 \bibitem[Ar]{Arai} K.\ Arai,    
{\it  Algebraic points on Shimura curves of  $\Gamma_0(p)$-type},
J. reine angew. Math. {\bf 690} (2014), 179--202


 \bibitem[Bo]{Bou} A.\ Bourdon,
    {\it A uniform version of a finiteness conjecture for CM elliptic curves}, Math.\ Res.\ Lett.\ {\bf 22}, (2015), 403--416


\bibitem[CL]{CL}I.\ Chen and Y.\ Lee,
{\it Explicit isogeny theorem for Drinfeld modules}, Pacific J.\ Math.\ {\bf 263}, no.1 (2013), 87--116 

\bibitem[Dr]{Dri}V.\ G.\ Drinfeld, 
{\it Elliptic modules},   Math. USSR\ Sub. {\bf 23} (1974), 561--592 

\bibitem[Fa]{Fal} G.\ Faltings,
   {\it Finiteness theorems for abelian varieties over number fields}, in: ``Arithmetic Geometry'', G.\ Cornell, J.\ H.\ Silverman(eds.), Springer-Verlag (1986), 9--27 

 \bibitem[Go]{Goss} D.\ Goss,
    {\it Basic structures of function field arithmetic},
  Ergebnisse der Mathematik und ihrer Grenzgebiete Volume 35, Springer-Verlag, Berlin (1996)

\bibitem[Ga1]{Gar} F.\ Gardeyn,   
{\it  t-motives and Galois representations},
  Ph.D.thesis, Universiteit Gent  (2001)


\bibitem[Ga2]{Gar1}  F.\ Gardeyn,   
{\it  A Galois criterion for good reduction of $\tau$-sheaves},
  J. Number Theory {\bf 97}  (2002), 447--471 

\bibitem[Ga3]{Gar2} F.\ Gardeyn,   
{\it  The structure of analytic $\tau$-sheaves},
  J. Number Theory {\bf 100}  (2003), 332--362 

\bibitem[Ga4]{Gar3} F.\ Gardeyn,   
{\it  Analytic morphisms of $t$-motives},
  Math.\ Ann.\  {\bf 325}  (2003), 795--828


  \bibitem[Ih]{Ihara} Y.\ Ihara,
   {\it Profinite braid groups, Galois representations and complex multiplications}, Ann.\ of Math. {\bf 123} (1986), 43--106

\bibitem[Ki]{Kim}W.\ Kim, 
{\it Galois deformation theory for norm fields and its arithmetic applications}, 
Thesis, The University of Michigan (2009)

\bibitem[KS]{KS}V. Kumer and J.\ Scherk,
{\it Effective versions of the Chebotarev density theorem for function fields},
C.\ R.\ Acad.\ Paris S\'er.\ I Math.\ {\bf 319}, no.6 (1994), 523--528

\bibitem[Lo]{L}D.\ Lombardo,
{\it On the uniform Rasmussen-Tamagawa conjecture in the CM case}, arXiv:1511.09019



%



\bibitem[Oz1]{Ozeki}Y.\ Ozeki, {\it Non-existence of certain Galois representations with a uniform tame inertia weight}, Int.\ Math.\ Res.\ Not.\  (2011), 2377--2395

 \bibitem[Oz2]{OzekiCM}Y.\ Ozeki,
    {\it Non-existence of certain CM abelian varieties with prime power torsion}, 
    Tohoku Math.\ J.\ {\bf 65}  (2013),  357--371





%




\bibitem[Ro]{Ros} M.\ Rosen,
{\it Number Theory in Function Fields}, 
Graduate Texts in Mathematics {\bf 210}, Springer-Verlag, New York (2002)

 \bibitem[RT1]{RT1} C.\ Rasmussen and A.\ Tamagawa,
    {\it A finiteness conjecture on abelian varieties with constrained prime power torsion}, Math.\ Res.\ Lett.\ {\bf15}  (2008), 1223--1231  


\bibitem[RT2]{RT2} C.\ Rasmussen and A.\ Tamagawa,
    {\it Arithmetic of abelian varieties with constrained torsion}, Trans.\ Amer.\ Math.\ Soc. {\bf 369}   (2017), 2395--2424


\bibitem[Se]{Serre}
J.-P.\ Serre, {\it Propri\'et\'es galoisiennes des points d'order fini des courbes elliptiques}, 
Invent.\ Math.\ {\bf 15}  (1972), 259--331

\bibitem[Se]{Serre2}J.-P.\ Serre, 
{\it Local Fields}, Graduate Texts in Mathematics {\bf 67}, Springer, New York (1979)

\bibitem[ST]{ST} J.-P.\ Serre and J.\ Tate,
   {\it Good reduction of abelian varieties}, Ann.\ of Math.\ {\bf88}  (1968), 492--517 
   



  \bibitem[Ta]{Tak}T.\ Takahashi, 
{\it Good reduction of elliptic modules}, J.\ Math.\ Soc.\ Japan {\bf 34}, no.3 (1982),
475--487

\end{thebibliography}
\end{document}